\documentclass[11pt]{article}
\usepackage[T1]{fontenc}
\usepackage[utf8]{inputenc}
\usepackage{lmodern,amsmath,amssymb,amsthm,dsfont}
\usepackage[margin=1in]{geometry}
\usepackage{microtype,hyperref,enumitem}
\usepackage[capitalise]{cleveref} 
\usepackage{algorithm}
\usepackage{algpseudocode}
\usepackage{caption}
\hypersetup{colorlinks=true,linkcolor=blue,citecolor=blue,urlcolor=blue,
pdftitle={Proof of the Linear Hadwiger's conjecture}}
\newtheorem{theorem}{Theorem}
\newtheorem{lemma}[theorem]{Lemma}
\newtheorem{proposition}[theorem]{Proposition}
\newtheorem{corollary}[theorem]{Corollary}
\newtheorem{conjecture}[theorem]{Conjecture}
\newtheorem{observation}[theorem]{Observation}
\newtheorem{claim}{Claim}[theorem]

\newcommand{\had}{h}
\renewcommand{\v}{\textup{\textsf{v}}}
\newcommand{\e}{\textup{\textsf{e}}}
\newcommand{\mc}[1]{\mathcal{#1}} 
\newcommand{\bb}[1]{\mathbb{#1}} 
\newcommand{\s}[1]{\left(#1\right)}
\newcommand{\eps}{\varepsilon}
\newcommand{\deff}{\operatorname{def}}

\newtheorem*{bootstraptheorem}{Theorem~\ref{thm:bootstrap}}

\title{A Proof of the Linear Hadwiger Conjecture}
\author{
Sergey Norin\,\footnotemark[1] 
\qquad Raphael Steiner\,\footnotemark[2]
}
\date{}
\begin{document}
\maketitle

\begin{abstract}
We show that there exists $C\in\mathbb{N}$ such that $K_t$-minor free graphs are $Ct$-colorable.

The proof was found by GPT-6 Astra, following the directions by the authors.

\end{abstract}

\footnotetext[1]{Department of Mathematics and Statistics, McGill University, Montr\'eal, Canada (\textsf{\href{mailto:sergey.norin@mcgill.ca}{sergey.norin\allowbreak@mcgill.ca}}).  Research supported by NSERC.}

\footnotetext[2]{Department of Mathematics, ETH Zürich, Switzerland (\textsf{\href{mailto:raphaelmario.steiner@math.ethz.ch}{raphaelmario.steiner@math.ethz.ch}}). Research supported by the SNSF Ambizione Grant No. 216071.}

\section{Introduction}

Hadwiger's famous coloring conjecture from 1943 is amongst the oldest and most fundamental problems in graph theory, claiming a deep relationship between the \emph{chromatic number} $\chi(G)$ of a graph $G$ and its \emph{Hadwiger number} $h(G)$, i.e., the largest integer $h$ such that $G$ has a $K_h$-minor. 

\begin{conjecture}[Hadwiger~\cite{Had}]\label{con:hadwiger}
Every graph $G$ satisfies $\chi(G)\le h(G)$. 
\end{conjecture}

Despite tremendous efforts by leading experts in the area for decades, Hadwiger's conjecture remains far from being resolved. Since it would be impossible to survey the vast landscape of partial results on this conjecture here, we refer to the survey articles~\cite{MR1411244,MR3526944,MR4680416} on the subject, and restrict ourselves here to only a handful of selected partial results. 

Hadwiger's conjecture has been proved for all graphs $G$ such that $h(G)\le 5$ and remains open when $h(G)\ge 6$. The cases $h(G)\in \{1,2\}$ are trivial, the case $h(G)=3$ was proved independently by Hadwiger himself~\cite{Had} and Dirac~\cite{MR45371}. The first substantially hard case of the conjecture is $h(G)=4$, when it becomes equivalent to the Four-Color-Theorem as demonstrated by Wagner~\cite{Wag}. The next case $h(G)=5$ was also reduced to the Four-Color-Theorem in a tour de force by Robertson, Seymour and Thomas~\cite{RST}, who managed to show that a smallest counterexample to this case of the conjecture would, upon the deletion of a single vertex, yield a counterexample to the Four-Color-Theorem. As the Four-Color-Theorem has been proven~\cite{MR543792,MR543793}, both of these special cases of the conjecture are known to hold as well. 

Due to the apparent difficulty of the precise version of the conjecture, researchers have naturally turned to the investigation of asymptotic results, i.e., trying to find a function $f:\mathbb{N}\rightarrow \mathbb{N}$ such that every graph $G$ satisfies $\chi(G)\le f(h(G))$, and trying to make the asymptotic growth of $f$ as slow as possible. For several decades, the best result of this form was that one can take $f(h)=\Theta(h\sqrt{\log h})$, which followed directly via the degeneracy-upper bound on the chromatic number from the seminal results of Kostochka~\cite{Kos} and Thomason~\cite{Tho} on the maximum average degree of graphs excluding a clique minor. Only a few years ago, the first author, Postle and Song~\cite{NPS} broke this degeneracy barrier by proving that one can take $f$ to be of the form $f(h)=h\log^{1/4+o(1)}h$. Further improving this result, Delcourt and Postle~\cite{DP} proved that one can take $f(h)=\Theta(h\log\log h)$, but the even bigger breakthrough here was that they also showed that the problem can be reduced down to graphs of small order (see Theorem~\ref{thm:DP} below). Building on this key conceptual step, recently the bound has been improved further by Liu and Luo~\cite{LL}, showing that one can take $f(h)=\Theta(h\log\log\log h)$. In this paper, we complete this line of research by presenting a proof that $f(h)$ can be taken to be linear in $h$. This is clearly asymptotically best-possible and resolves a well-known folklore conjecture (see e.g.~\cite{RS,KawMoh06}) in the area, known as the \emph{Linear Hadwiger conjecture}.

\begin{theorem}\label{thm:linear}
There is an absolute constant $C$ such that every graph $G$ satisfies $\chi(G)\le C\cdot h(G)$.
\end{theorem}
 
The proof of Theorem~\ref{thm:linear} which we present here builds heavily on the aforementioned breakthrough-work of Delcourt and Postle~\cite{DP}, who showed that to prove the Linear Hadwiger conjecture, it suffices to show it for graphs whose order is at most a polylogarithmic factor above their Hadwiger number. More precisely, the following statement follows immediately from the main result of their paper\footnote{In fact, their main result justifies any exponent strictly larger than $4$ on top of the logarithmic factor in the following statement, and the choice of $5$ is out of convenience.}.

\begin{theorem}[Corollary of Theorem~1.6 of~\cite{DP}]\label{thm:DP}
If there exists a constant $C\ge 1$ such that for every sufficiently large $t\in \mathbb{R}_{\ge 1}$, every graph $G$ with $h(G)<t$ and $\v(G)\le t\log^5 t$ satisfies $\chi(G)\le Ct$, then there also exists a constant $C'\ge 1$ such that for every $t\in\mathbb{R}_{\ge 1}$, every graph $G$ with $h(G)<t$ satisfies $\chi(G)\le C't$.
\end{theorem}

To simplify the following discussion, we denote by $\mathrm{LH}(\alpha)$ for $\alpha > 0$ the statement that there exists an absolute constant $C_\alpha\ge 1$ such that for every sufficiently large $t\in \mathbb{R}_{\ge 1}$, every graph $G$ with $h(G)<t$ and $\v(G)\le t\log^{\alpha} t$ satisfies $\chi(G)\le C_\alpha t$.

By Theorem~\ref{thm:DP}, to prove Theorem~\ref{thm:linear}, it suffices to prove $\mathrm{LH}(5)$. We will achieve this in two clearly separate steps: First, we show that the Linear Hadwiger conjecture holds for graphs of very small order, namely graphs of order at most $t\log^{1/2-\varepsilon}t$, with a surprisingly small constant.

\begin{theorem}\label{thm:small}
Let $\varepsilon\in (0,1/2)$ be fixed. For every sufficiently large $t\in \mathbb{R}_{\ge 1}$, every graph $G$ with $h(G)<t$ and $\v(G)\le t(\log t)^{1/2-\varepsilon}$ has chromatic number at most $(4+\varepsilon)t$. 
\end{theorem}

In particular, this establishes that $\mathrm{LH}(\alpha)$ holds for every $0<\alpha<\frac{1}{2}$. 

Second, we show that one can bootstrap from this initial range to the full range that needs to be covered to prove $\mathrm{LH}(5)$. Concretely, one can boost the exponent on top of the logarithmic factor in the graph order by a constant, at the expense of increasing the constant factor in the bound for the chromatic number.

\begin{theorem}\label{thm:bootstrap}
For every $\alpha>0$, we have that $\mathrm{LH}(\alpha)$ implies $\mathrm{LH}(\frac{4}{3}\alpha)$.
\end{theorem}

It is easy to see that Theorems~\ref{thm:small} and~\ref{thm:bootstrap} together with Theorem~\ref{thm:DP} prove the Linear Hadwiger conjecture, as formally written out in the following proof.

\begin{proof}[Proof of Theorem~\ref{thm:linear}, assuming Theorems~\ref{thm:small} and~\ref{thm:bootstrap}] 
By Theorem~\ref{thm:DP}, to prove Theorem~\ref{thm:linear} it suffices to show that $\mathrm{LH}(5)$ holds. By Theorem~\ref{thm:small}, we know that, say, $\mathrm{LH}(\frac{1}{3})$ holds. Furthermore, by repeated application of Theorem~\ref{thm:bootstrap}, we can obtain from this via the logical chain
$$\mathrm{LH}\left(\frac{1}{3}\right) \Rightarrow \mathrm{LH}\left(\frac{1}{3}\cdot \frac{4}{3}\right) \Rightarrow \mathrm{LH}\left(\frac{1}{3}\cdot \left(\frac{4}{3}\right)^2\right)\Rightarrow \cdots \Rightarrow \mathrm{LH}\left(\frac{1}{3}\cdot \left(\frac{4}{3}\right)^{10}\right)$$ that also $\mathrm{LH}\left(\frac{1}{3}\cdot \left(\frac{4}{3}\right)^{10}\right)$ holds. Since $\frac{1}{3}\cdot \left(\frac{4}{3}\right)^{10}\approx 5.92\ge 5$, this immediately implies that also $\mathrm{LH}(5)$ and thus Theorem~\ref{thm:linear} holds, concluding the proof.
\end{proof}

\paragraph*{Organization.}

The rest of this paper is divided into two independent halves (Sections~\ref{sec:rapha} and~\ref{sec:bootstrap}), in which we, respectively, prove Theorems~\ref{thm:small} and~\ref{thm:bootstrap}. At the beginning of each section, we give a high-level overview of the key ideas and the structure of the proofs of each of these two theorems, which we suggest to be read first by the reader to have a high-level route and intuition in mind while reading the later subsections, in which we supply more of the technical innovations and details.

Each of the two halves has technical ingredients that mostly just rework previously existing proofs, namely: \begin{itemize} \item [(1)] a quantitative version of a well-known theorem of Pippenger, Frankl and R\"{o}dl~\cite{FR,P} on almost spanning matchings in near-regular hypergraphs with small codegrees (Theorem~\ref{thm:pippenger}) which seems to not have explicitly appeared in writing with this quantitative dependence of parameters (but follows implicitly from the existing proofs), and \item [(2)] a direct consequence of the proofs of Delcourt and Postle~\cite{DP} based on their chromatic separability framework (Theorem~\ref{thm:DPinsep}), essentially reducing the task of coloring a given clique-minor-free graph $G$ to studying its so-called ``chromatic inseparable'' induced subgraphs.\end{itemize} Since both of these results have implicitly but not explicitly existed in the literature before, we have decided to move the proofs of both Theorems~\ref{thm:pippenger} and~\ref{thm:DPinsep} to separate sections in the appendix of this paper. In Section~\ref{sec:conc} we end with some concluding remarks, including an additional consequence of our main result.

\paragraph*{AI disclosure and discussion.} The proof presented in this paper was developed by OpenAI’s ChatGPT model Astra 6, following research directions and feedback supplied by the authors. The authors suggested several possible approaches, but almost none of their specific proof ideas were retained in the final argument. (Their suggestion to further exploit Liu and Luo’s  ideas of contractions involving unusually large independent sets in vertex neighbourhoods is reflected in the contraction argument used in the proof of Theorem~\ref{thm:bootstrap}.)

The project began with the authors asking the model to prove the very dense case of Linear Hadwiger’s conjecture, which they believed was the crucial missing piece of the proof. Concretely, the model was asked to prove that there exists an absolute constant $C>0$ such that, for every $C'>0$, there exists $t_0>0$ such that, for every integer $t\ge t_0$, every graph $G$ satisfying
$
\v(G)\le C't$ and
$
\chi(G)>Ct
$
contains a $K_t$ minor. 

The model surprised the authors by succeeding after only a couple of hours and some encouragement. We then asked it to make the dependence of $t_0$ on $C'$ explicit.
This request yielded a proof of the linear Hadwiger conjecture for graphs $G$ with $\had(G)=t$ and $\v(G) \leq t(\log t)^{1/2-o(1)}$, while by the Delcourt-Postle result establishing it for graphs with $\v(G) \leq t(\log t)^{5}$ (in fact, $t(\log t)^{2+o(1)}$ by a later improved result of Lin~\cite{Lin}) suffices. 

At this point, we became convinced that the full proof is well within reach.
Several approaches were tried. We asked the model, in particular,  to find an original idea to bridge the gap between the two bounds, rather than trying to tighten them. This led to the bootstrap argument used in the present paper.

A posteriori, the linear Hadwiger conjecture appears to be a perfect problem for the current\footnote{as of late September 2026} generation of AI models to tackle.
A large corpus of research on the problem was available for the model to take inspiration from. The proof presented in this paper is, figuratively, contained in the convex hull of existing results. Yet it does not lie on a low-dimensional face of it. Indeed, the methods used in the proof include \begin{itemize}
\item Reed-Seymour's fractional linear Hadwiger theorem~\cite{RS},
\item Delcourt-Postle's reduction of the linear Hadwiger conjecture to small graphs~\cite{DP},
\item Delcourt-Postle's chromatic separability framework, used in the above reduction~\cite{DP},
\item Pippenger-Frankl-R\"{o}dl's matching theorem for large-degree-small-codegree hypergraphs~\cite{P,FR},
\item LP duality, in particular an analysis of (to the best of our knowledge) novel variants of the standard LP-s~\cite{MR1481157} for the fractional chromatic number of a graph which additionally control the \emph{simultaneous} coverage of \emph{pairs} of vertices by independent sets,
\item Gy\'{a}rf\'{a}s's path argument~\cite{Gyarfas}, and
\item the factor-two lifting of a coloring from a minor obtained by contracting connected bipartite pieces  \`a la ~\cite{RS,MR4379303}.
\end{itemize}
(Surprisingly, the key new ideas of the current proof don't rely on density-based arguments unlike all the earlier improvements of the asymptotic bounds.)

The writing of this paper is due to the human authors. In particular, a lot of effort was put into explaining the high-level proof strategies and intuitions for the new ideas in the paper\footnote{A curious reader may want to compare with an early AI-written draft which can be accessed on the GitHub page  \url{https://github.com/snorin239/LinearHadwiger/blob/master/paper/LinearHadwigerAIDraft.pdf}.}. ChatGPT 6-Pro was used for proofreading and typo-checking, and it helped with the creation of the bibliography. The authors take full responsibility for the quality of exposition, the credit to previous work, and the mathematical correctness of all the presented arguments.

\paragraph*{Lean formalization.} A Lean formalization, of the full proof presented here as well as its external inputs, was produced by OpenAI Codex versions 6 Sol and Astra. It is available at \url{https://github.com/snorin239/LinearHadwiger}.

\paragraph*{Acknowledgements.} The first author thanks Neil Rahman for helpful discussions of ~\cite{LL}. The second author would like to thank Micha Christoph, Patryk Morawski and Benny Sudakov for interesting discussions on the topic around a year ago.

\paragraph*{Notation.} Given an integer $n\ge 0$, we use $[n]$ as a shorthand for the set $\{1,\ldots,n\}$ of the first $n$ natural numbers ($[0]:=\emptyset$). For a real number $x$, we use the notation $x^+:=\max\{x,0\}$ for the positive part of $x$.
For a finite set $X$ and an integer $k\ge 0$, we denote by
$\binom{X}{k}$ the collection of all $k$-element subsets of $X$.
All logarithms are natural.

All graphs considered are finite and simple. Given a graph $G$,
we denote its vertex-set and edge-set by $V(G)$ and $E(G)$,
respectively, and write $\v(G):=|V(G)|$ for its order.
We denote by $\Delta(G)$ its maximum degree and by $\alpha(G)$
its independence number. Given a subset $X\subseteq V(G)$,
we write $G[X]$ for the subgraph of $G$ induced by $X$, and
$G-X$ or $G\setminus X$ for the graph obtained from $G$ by deleting the vertices in $X$. For a subset of edges $F\subseteq E(G)$, we denote by $G/F$ the graph obtained from $G$ by contracting all edges in $F$ and omitting any loops arising from the contractions and removing duplicate edges afterwards.

For a vertex $v\in V(G)$, we denote its neighborhood by $N_G(v)$
and its closed neighborhood by $N_G[v]:=N_G(v)\cup\{v\}$.
For a subset $X\subseteq V(G)$, we similarly write
$N_G[X]:=\bigcup_{v\in X}N_G[v]$.

For graphs $G_1$ and $G_2$ on the same vertex-set, $G_1\cup G_2$ denotes
the graph on vertex-set $V(G_1)=V(G_2)$ with edge-set $E(G_1)\cup E(G_2)$.

A hypergraph $\mathcal H$ is a tuple $(V,E)$ where $V$ is a finite set of vertices and $E\subseteq 2^V\setminus \{\emptyset\}$ is the set of \emph{hyperedges}. In a multi-hypergraph, we allow $E$ to be a multi-set of non-empty subsets of $V$. As for graphs, we use the notation $V(\mathcal H):=V$ and $E(\mathcal H):=E$ for the vertex- and hyperedge-set, respectively, and use $\v(\mathcal H)$ as a shorthand for the number of vertices of $\mathcal H$. We say that $\mathcal H$ is \emph{$r$-uniform} for some integer $r\ge 2$ if $E(\mathcal H)\subseteq \binom{V(\mathcal H)}{r}$. The \emph{degree} of a vertex in a (multi-)hypergraph is the number of hyperedges containing it, and the \emph{codegree} of a pair $\{u,v\}$ of two distinct vertices is the number of hyperedges containing both $u$ and $v$ at the same time. A similar definition applies to weighted hypergraphs, where hyperedges are assigned non-negative real weights: Then the weighted (co)degree is defined analogously, except that we consider the total weight of the relevant hyperedges rather than their total number.

\section{Coloring very small $K_t$-minor free graphs with $(4+o(1))t$ colors}\label{sec:rapha}
In this section, we prove Theorem~\ref{thm:small}. Concretely, we wish to show that for every sufficiently large~$t\in \mathbb{R}_{\ge 1}$, every graph $G$ with $h(G)<t$ and $\v(G)\le t(\log t)^{1/2-\varepsilon}$ has chromatic number at most $(4+\varepsilon)t$. Before delving into details, we first give a high-level overview of this proof in the following subsection. During this high-level overview, we already state and explain the two key lemmas in context and give the proof of Theorem~\ref{thm:small} based on these lemmas. Then, we present intuition for and details of the proofs of the key lemmas in two dedicated subsections, namely Subsections~\ref{sec:rapha-rounding} and~\ref{sec:rapha-bounding}.

\subsection{Proof overview}\label{sec:rapha-overview} A simple observation is that from the get-go, we may focus ourselves on graphs $G$ as above which additionally have small independence number; concretely, graphs $G$ as above for which $\alpha(G)\le (\log t)^{1/2-\varepsilon/2}$. The reason for this is simple: As long as an independent set larger than $(\log t)^{1/2-\varepsilon/2}$ exists in the current graph, we can simply pull out such a large independent set, assign its vertices one separate color, and delete it from the graph. Repeating this process as long as possible, we eventually reach an induced subgraph $G'$ of our initial graph $G$ which satisfies $\alpha(G')\le (\log t)^{1/2-\varepsilon/2}$. Crucially, the number of independent sets we pulled out to reach this subgraph $G'$ can never be larger than $\frac{\v(G)}{(\log t)^{1/2-\varepsilon/2}}\le o(t)$, meaning that if we can color $G'$ with at most $(4+\varepsilon/2)t$ colors, then we can color $G$ with at most  $(4+\varepsilon)t$ colors, as desired. Hence, possibly by replacing $\varepsilon$ by $\varepsilon/2$, from now on we may and will assume that the graph $G$ we wish to color by at most $(4+\varepsilon)t$ colors satisfies $\alpha(G)\le (\log t)^{1/2-\varepsilon}, h(G)<t$ and $\v(G)\le t(\log t)^{1/2-\varepsilon}$.

The real starting point of the proof however, and one of the key ideas, (which has already appeared in previous work on the subject~\cite{LL}) is to use the fact that Linear Hadwiger's conjecture, when relaxed to fractional coloring, has been confirmed with the very small constant $2$ by Reed and Seymour already in 1998.
\begin{theorem}[Reed-Seymour~\cite{RS}]\label{thm:RS}
 Every graph $G$ satisfies $\chi_f(G)\le 2\cdot h(G).$
\end{theorem}
Here, as usual, $\chi_f(G)$ denotes the \emph{fractional chromatic number}, which can be defined as the optimal value of the following linear program. We denote by $\mathcal{I}(G)$ the collection of independent vertex subsets in $G$.

\begin{align*}
\tag{P}
    \text{min} \sum_{I \in \mathcal{I}(G)}&{x_I} \\
    \text{s.t.} \sum_{I\in \mathcal{I}(G): v \in I}&{x_I}\ge 1~~(\forall v \in V(G)), \\
    & x_I \ge 0~~(\forall I \in \mathcal{I}(G)).
\end{align*}

In fact, it is well-known and easy to see that requiring equality in the covering constraints does not change the optimal value, namely $\chi_f(G)$ is also the optimal value of the following LP:

\begin{align*}
\tag{P'}
    \text{min} \sum_{I \in \mathcal{I}(G)}&{x_I} \\
    \text{s.t.} \sum_{I\in \mathcal{I}(G): v \in I}&{x_I}= 1~~(\forall v \in V(G)), \\
    & x_I \ge 0~~(\forall I \in \mathcal{I}(G)).
\end{align*}

By LP duality, $\chi_f(G)$ also equals the optimal value of the following linear dual program of~(P). 

\begin{align*}
\tag{D}
    \text{max} &\sum_{v\in V(G)}{y_v} \\
    \text{s.t.}~~~~&\sum_{v \in I}{y_v}\le 1~~(\forall I \in \mathcal{I}(G)), \\
    &~~~~~y_v \ge 0~~(\forall v\in V(G)).
\end{align*}

Notice that the integer linear program (IP) versions of the primal programs~(P) and~(P') have optimal values exactly $\chi(G)$. Hence, if one could somehow obtain an integer feasible solution of the primal program from an optimal fractional solution by a suitable rounding procedure that does not increase the optimal value by more than a constant factor, then one would obtain a linear upper bound on $\chi(G)$ in terms of $h(G)$ by using Reed and Seymour's aforementioned result. However, this is generally not possible, since it is well-known that in general the ratio between $\chi(G)$ and $\chi_f(G)$ can be of order $\Theta(\log \v(G))$ (as shown, for example, by suitable Kneser graphs), which even in the small order range for $G$ we care about here would be too big of a loss. To circumvent this issue, the first step in the proof of Theorem~\ref{thm:small} consists in showing that if we allow ourselves some more leeway, we can indeed relate ordinary and fractional chromatic number. Here ``some more leeway'' means that we allow to bound $\chi(G)$ not just in terms of $\chi_f(G)$ but instead by $\chi_f(G')$, where $G'=G\cup D$ is a supergraph of $G$ on the same vertex-set as $G$ which is obtained from $G$ by adding some additional edges, forming a sparse, bounded-degree graph $D$. Furthermore, we allow an additive error term depending on the order of the graph, which will be fine here since we discuss graphs of very small order. 

Concretely, we prove the following general graph-theoretic statement that stands on its own and does not directly relate to graph minors. The parameter $\gamma$ is an auxiliary parameter which allows a tradeoff between the maximum degree of the graph $D$ used for augmentation and the size of the additive error term for the rounding.

\begin{lemma}\label{lem:rounding}
Let $n\in\mathbb{N}$, $m\in \mathbb{R}_{\ge 2}$ and $\gamma\in (0,1)$ be such that $n\ge \left(\frac{600(m+2)}{\gamma}\right)^{20m+40}$ and let $G$ be any graph on $n$ vertices with $\alpha(G)\le m$.  There exists a graph $D$ on the same vertex-set as $G$ with $$\chi(G)\le \chi_f(G\cup D)+\gamma n \text{ and }\Delta(D)\le \left(\frac{600(m+2)}{\gamma}\right)^{20m+41}.$$
\end{lemma}

We defer a discussion of the interesting new ideas going into the proof of Lemma~\ref{lem:rounding} to the start of the dedicated Subsection~\ref{sec:rapha-rounding} where this lemma is proved.

Note that if we apply Lemma~\ref{lem:rounding} to our setting, then we will naturally set $m$ as $\log^{1/2-\varepsilon} t$, and it makes sense to choose $\gamma$ as $\frac{\varepsilon}{2}\cdot \frac{1}{\log^{1/2-\varepsilon} t}$, as this will ensure that the additive error term in the above estimate is at most $\frac{\varepsilon}{2}\cdot t$ and can thus be neglected for our purposes. This will put the maximum degree bound for $D$ roughly at the order of $e^{O(\log^{1/2-\varepsilon} t \cdot \log\log t})<e^{\log^{1/2} t}$. 

The key question then becomes how to upper bound $\chi_f(G\cup D)$ when all we know about $G$ is that it has Hadwiger number $h(G)<t$, that $\alpha(G)\le m=\log^{1/2-\varepsilon} t$ and that $D$ has degree bounded by $e^{\log^{1/2-\varepsilon/2} t}$. At first look, this may seem a bit hopeless: The naive product bound, using the aforementioned result of Reed and Seymour~\cite{RS}, only yields $\chi_f(G\cup D)\le \chi_f(G)\chi_f(D)\le 2t\cdot (e^{\log^{1/2} t}+1)$ which is much larger than the targeted linear bound for the chromatic number. This is where the second key idea of this proof is needed, which shows that by combining the small independence number of $G$, the degree-bound on $D$ as well as by applying Reed and Seymour's fractional linear Hadwiger bound to \emph{minors} of $G$ rather than just $G$ itself, one can in fact significantly improve over the naive product bound. Concretely, Reed and Seymour's result is applied to suitable minors of $G$ obtained by contracting cleverly chosen matchings. Formally, we prove the following result.

\begin{lemma}\label{lem:bounding}
Let $m, d, t\in \mathbb{R}_{\ge 1}$. Let $G$ be a graph with $h(G)<t$ and $\alpha(G)\le m$, and let $D$ be a graph on the same vertex-set as $G$ such that $\Delta(D)\le d$. Then
$$\chi_f(G\cup D)\le 4t+(d+1)^{4m}.$$
\end{lemma}

Again, we defer the intuition and high-level overview of the ideas going into the proof of this lemma to the beginning of the dedicated Subsection~\ref{sec:rapha-bounding}.

Having shown both key lemmas, one can now easily conclude the proof: The term $4t$ coming from the upper bound on $\chi_f(G\cup D)$ in the previous lemma already has the correct order, so to prove our desired bound of $(4+\varepsilon)t$ on the total chromatic number all we need is that the second term $(d+1)^{4m}$ is sublinear in $t$. Plugging in our bounds $m=(\log t)^{1/2-\varepsilon}$ and $d=e^{\log^{1/2} t}$ justified by our previous discussion, we can bound this term as $(d+1)^{4m}\le e^{O(m\log d)}=e^{O((\log^{1/2-\varepsilon} t)\cdot (\log^{1/2} t))}=e^{O(\log^{1-\varepsilon} t)}$, which is certainly sublinear. Hence, we indeed obtain that the second additive term is negligible, giving the desired upper bound of $(4+\varepsilon)t$ for the chromatic number of our initial graph for sufficiently large $t$, and proving Theorem~\ref{thm:small}, as desired. In retrospect, it is now easy to see why we chose $t\cdot (\log t)^{1/2-\varepsilon}$ as the cutoff on the graph order in Theorem~\ref{thm:small}: Putting the two lemmas together and analyzing the constraints necessary to obtain a linear bound on the chromatic number in terms of~$t$, one can see that the threshold where the error terms become too large to handle is exactly around $t\cdot (\log t)^{1/2-o(1)}$. We can also see that it is crucial to obtain (at least asymptotically) a dependency of the maximum degree given by Lemma~\ref{lem:rounding} on $\gamma$ which is not too large, for otherwise the same argument may push the cut-off in Theorem~\ref{thm:small} to $t(\log t)^{o(1)}$, which would make it impossible to bootstrap the result using Theorem~\ref{thm:bootstrap} further to cover the full range of relevant graph orders necessary to prove Theorem~\ref{thm:linear}.

Let us put the outline presented so far together in a formal proof of Theorem~\ref{thm:small} based on the key Lemmas~\ref{lem:rounding} and~\ref{lem:bounding}, before we proceed to presenting the intuition for and proofs of these two key lemmas in the next two subsections.

\begin{proof}[Proof of Theorem~\ref{thm:small}, assuming Lemmas~\ref{lem:rounding} and~\ref{lem:bounding}]
Let $\varepsilon\in (0,1)$ be fixed and let $t\in \mathbb{R}_{\ge 1}$ be a sufficiently large real number (throughout assumed to be large enough for any estimates used in the following which hold only for sufficiently large $t$ to be satisfied for $t$). 

Let $G$ be any given graph with $h(G)<t$ and $\v(G)\le t(\log t)^{1/2-\varepsilon}$. Our goal is to show that $G$ has chromatic number at most $(4+\varepsilon)t$.

Let $I_1,I_2,\ldots, I_k$ for some $k\in \mathbb{Z}_{\ge 0}$ be a sequence of disjoint independent sets, each of size at least $(\log t)^{1/2-\varepsilon/2}$, built by the following greedy process: For $i\ge 1$, having already defined all $I_j, 1\le j<i$, if $\alpha\left(G-\bigcup_{1\le j<i}I_j\right)\ge (\log t)^{1/2-\varepsilon/2}$, then we define $I_i$ as a maximum-size independent set in $G-\bigcup_{1\le j<i}I_j$. Otherwise, we stop the process and set $k:=i-1$. It is easy to verify that the induced subgraph $G':=G-(I_1\cup \cdots \cup I_k)$ left at the end of the process has $\alpha(G')\le (\log t)^{1/2-\varepsilon/2}=:m$ and $$\chi(G)\le \chi(G')+k\le \chi(G')+\frac{\v(G)}{(\log t)^{1/2-\varepsilon/2}}\le \chi(G')+\frac{t}{\log^{\varepsilon/2} t}<\chi(G')+\frac{\varepsilon}{3}t.$$

Thus, in what follows it suffices to show that $\chi(G')\le \left(4+\frac{2\varepsilon}{3}\right)t$. Since we are immediately done with this when $\v(G')<t$, we may assume $\v(G')\ge t$ in what follows. We first apply Lemma~\ref{lem:rounding} to $G'$ with $m$ as defined above, and $\gamma:=\frac{\varepsilon}{2\log^{1/2-\varepsilon}t}$. Note that this is feasible, since one can check that

$$\left(\frac{600(m+2)}{\gamma}\right)^{20m+40}=e^{\Theta(\log^{1/2-\varepsilon/2} t\log\log t)} <t\le \v(G')$$ with our choices of $m$ and $\gamma$ and for every sufficiently large $t$.

We thus obtain a graph $D$ on the same vertex-set as $G'$ such that 
$$\chi(G')\le \chi_f(G'\cup D)+\gamma \v(G')\le\chi_f(G'\cup D)+\frac{\varepsilon}{2}t,$$
(where we used our definition of $\gamma$ and that $\v(G')\le \v(G)\le t \log^{1/2-\varepsilon} t$), as well as
$$\Delta(D)\le \left(\frac{600(m+2)}{\gamma}\right)^{20m+41}\le \exp(50m\log\log t)\le \exp((\log t)^{1/2})=:d,$$ where we used that $t$ is chosen sufficiently large.

Next, we apply Lemma~\ref{lem:bounding} to $G'$ and $D$ with parameters $m$ and $d$ and obtain that
$$\chi(G)\le \chi(G')+\frac{\varepsilon}{3}t\le \chi_f(G'\cup D)+\frac{5\varepsilon}{6}t\le 4t+(d+1)^{4m}+\frac{5\varepsilon}{6}t$$
$$\le \left(4+\frac{5\varepsilon}{6}\right)t+\exp(4m\log (d+1))\le \left(4+\frac{5\varepsilon}{6}\right)t+\exp(8(\log t)^{1-\varepsilon/2})\le (4+\varepsilon)t,$$ as desired (we used the fact that $t$ is chosen sufficiently large). This concludes the proof.
\end{proof}

\subsection{Chromatic number vs. fractional chromatic number via augmentation}\label{sec:rapha-rounding}

In this section, we discuss and present the proof of the first key lemma towards the proof of Theorem~\ref{thm:small}, namely Lemma~\ref{lem:rounding}. Before going into the details of its proof, let us first motivate the lemma statement and proof on a high-level. 

To start with, consider just any graph $G$ on $n$ vertices with $\alpha(G)\le m$ for some parameter $m$ (which we think of as being fairly small compared to the target chromatic number and compared to~$n$), and fix an optimal fractional coloring of it, i.e., an optimal solution $(x_I)_{I\in \mathcal{I}(G)}$ of the primal linear program~(P') for the fractional chromatic number. By definition, this means $\sum_{I\in \mathcal{I}(G)}x_I=\chi_f(G)$ and $\sum_{I\in \mathcal{I}(G): v\in I}x_I=1$ for every $v\in V(G)$.

One interesting way to see this from a different perspective is then to think of the independent sets $I$ of $G$ that are supported in the fractional coloring (with $x_I>0$) as the edges of a hypergraph $\mathcal{H}$ on ground set $V(G)$, and of the variables $(x_I)_{I\in \mathcal{I}(G)}$ as a positive weighting of the edges of $\mathcal{H}$ of total weight $\chi_f(G)$. The constraints then tell us that the ``weighted degree'' of every vertex~$v$, i.e. the sum of the weights of incident hyperedges at $v$, equals $1$ for every vertex, so this weighted hypergraph is in some sense regular. Furthermore, since $\alpha(G)\le m$, all hyperedges have size at most $m$ and are thus not too large. To get the key idea across, let us even assume for now (we will later see how to remove this very idealistic assumption) that the hyperedges in $\mathcal{H}$ all have size $m$. In this setup, an ideal scenario which would allow us to ``round'' the optimal fractional solution to a proper coloring of $G$ at not much cost would be if we could find a very large, almost spanning, matching $M$ in the hypergraph $\mathcal{H}$, which covers all but at most $\gamma n$ of the vertices of $G$ for some small $\gamma\in (0,1)$. Indeed, then giving a separate color to each independent set in $M$ and coloring the outside vertices each with a unique color would show $\chi(G)\le |M|+\gamma n\le \frac{n}{m}+\gamma n=\frac{n}{\alpha(G)}+\gamma n\le \chi_f(G)+\gamma n$, which is the inequality we seek to prove in Lemma~\ref{lem:rounding} but with no augmentation needed, i.e. $D=\emptyset$. However, clearly this is way too optimistic: In general, regularity of a (weighted) uniform hypergraph does not guarantee such a large matching. However, a classical result due to Pippenger, Frankl and R\"{o}dl (often referred to as \emph{Pippenger's theorem})~\cite{P,FR}, states that for (multi-)hypergraphs one can actually indeed guarantee such almost-spanning matchings in (near-)regular uniform hypergraphs under the additional assumption that all the \emph{codegrees} are very small relative to the degrees, i.e., distinct vertices share only a negligible amount of their incident edges. 

To turn our dreams into provable facts, we would like to make use of this result. This needs to overcome several obstacles:
\begin{itemize}
    \item How can we apply the result to \emph{weighted} hypergraphs?
    \item How can we guarantee the codegree condition?
\end{itemize}
The first question actually has an easy solution. First of all, the known proofs of Pippenger's theorem also work for multi-hypergraphs, where we are allowed to repeat hyperedges arbitrarily many times. But then, given a weighted hypergraph such as our $\mathcal{H}$, we can simply take a very large number $N$ and replace every hyperedge of our weighted hypergraph $\mathcal{H}$ with weight $w$ by approximately $wN$ copies of itself. This will give rise to a multi-hypergraph with proportional actual degrees and codegrees, to which we then can apply Pippenger's theorem. Since the original hypergraph and the new multi-hypergraph created in this way have corresponding matchings and the same vertex-set, we can directly translate an almost spanning matching obtained in the multi-hypergraph back to an almost spanning matching in our weighted hypergraph.

While this reduction would suffice to deduce the desired result on a \emph{qualitative} level from the literature, for our purposes we will need very explicit \emph{quantitative} dependencies between the ratio of leftover vertices and the necessary control on the ratio between codegrees and degrees. Thus, it will be convenient for us to reprove Pippenger's theorem with these more precise dependencies here, and to do it directly in the indicated weighted setting that we need. The following is the formal statement that we are going to use.

\begin{theorem}\label{thm:pippenger}
Let $r\ge2$ be an integer, $0<\xi\le1$, and set
$\mu=(\xi/100)^{20r}$. Suppose $\mathcal H$ is an $r$-uniform hypergraph equipped with a non-negative weighting $(x_e)_{e\in E(\mathcal H)}$ of its edges, such that for every vertex $v\in V(\mathcal H)$ we have
$$\sum_{e\in E(\mathcal H): v\in e}x_e\in [1-\mu,1],$$ while for every pair of distinct vertices $u,v\in V(\mathcal H)$ we have
$$\sum_{e\in E(\mathcal H): \{u,v\}\subseteq e}x_e\le \mu.$$
Then there exists a matching $M$ in $\mathcal H$ covering all but at most $\xi \v(\mathcal H)$ vertices of $\mathcal H$.
\end{theorem}

The proof of Theorem~\ref{thm:pippenger} has no essential new ideas and basically just reworks existing proofs of Pippenger's theorem in the weighted setup (which, as explained, is actually equivalent to the non-weighted setup) while precisely tracking the quantitative dependencies. We thus defer its proof to Section~\ref{app:matching} of the Appendix.

The answer to the second question above leads us to the bounded-degree augmentation $D$ of our initial graph. The simple answer is that in general, we have no handle on the (weighted) codegrees of our auxiliary hypergraph $\mathcal{H}$ coming from the fractional coloring, and in general there is no guarantee why $\sum_{I\in \mathcal{I}(G):I\supseteq \{u,v\}}x_I$ would have to be small compared to $1$ for every pair of distinct vertices $u,v$. However, it turns out that one can obtain control on these weighted codegrees by the addition of a few new edges, collected in an augmentation-graph $D$ of bounded degree. This is probably the most important new and key idea going into the proof of Lemma~\ref{lem:rounding}. To bring this intuitive idea to fruition formally, the right way turns out to be to analyze a modified dual pair of LPs, whose primal is obtained from the program~(P) for the fractional chromatic number by simply adding constraints enforcing an upper bound $\mu>0$ for the weighted codegrees:

\begin{align*}
\tag{P-CODEGREE}
    \text{min} \sum_{I \in \mathcal{I}(G)}{x_I} \\
    \\
    \text{s.t.}~~~~~~~\sum_{I\in \mathcal{I}(G): v \in I}&{x_I}\ge 1~~(\forall v \in V(G)), \\
    \sum_{I\in \mathcal{I}(G): \{u,v\} \subseteq I}&{x_I}\le \mu~~(\forall u\neq v \in V(G), uv\notin E(G)), \\
    & x_I \ge 0~~(\forall I \in \mathcal{I}(G)).
\end{align*}

The dual program then becomes:

\begin{align*}
\tag{D-CODEGREE}
    \text{max} \sum_{v\in V(G)}{y_v}-\mu&\sum_{\substack{uv\in\binom{V(G)}{2},\\ uv\notin E(G)}}z_{uv} \\
    \\
    \text{s.t.}~~~~\sum_{v \in I}{y_v}-\sum_{uv\in\binom{I}{2}}z_{uv}&\le 1~~(\forall I \in \mathcal{I}(G)), \\
  y_v &\ge 0~~(\forall v\in V(G)), \\
   z_{uv}&\ge 0~~\left(\forall uv\in \binom{V(G)}{2}, uv\notin E(G)\right).
\end{align*}

It turns out that a simple argument shows that in an optimal dual solution to~\textsc{(D-CODEGREE)} the sum of dual variables $z_{uv}$ incident at any fixed vertex $u$ is bounded by $\frac{1}{\mu}$, and hence, if we include into $D$ all those non-edge pairs $u,v$ for which $z_{uv}$ exceeds a certain threshold, we can ensure that $D$ has bounded degree. Picking pairs $u,v$ with a large dual variable into $D$ is a natural choice: Then any independent set $I$ in the graph $G\cup D$ only includes non-edge pairs $u,v$ where the dual variable $z_{uv}$ in the above optimal dual solution was small. The constraints then give $\sum_{v \in I}{y_v}\le 1+\sum_{uv\in\binom{I}{2}}z_{uv}$, which (for a suitably small choice of what it means for a dual variable to be small) can be bounded by $1+\delta$ for some small $\delta>0$. Hence, after rescaling the dual variables $y_v$ by $\frac{1}{1+\delta}$ we create a feasible dual solution for the ordinary dual program~(D) of the augmented graph $G\cup D$. This way, we can give a lower bound on $\chi_f(G\cup D)$ in terms of the optimal value of the program~\textsc{(D-CODEGREE)} for the original graph $G$. In particular, by turning to the primal program~\textsc{(P-CODEGREE)} for $G$ again, we will obtain a feasible fractional coloring of $G$ whose total weight is controlled in terms of $\chi_f(G\cup D)$ with all codegrees bounded by $\mu\ll 1$ (in the weighted sense we have been discussing previously). Having reached this desired situation, we can now apply the aforementioned strategy via the weighted Pippenger theorem (Theorem~\ref{thm:pippenger}) to obtain a decent upper bound on the chromatic number of our original graph $G$ in terms of $\chi_f(G\cup D)$ while only incurring a controllable error in both of the steps.

This concludes the description of the intuition and key novel ideas going into the proof of Lemma~\ref{lem:rounding}. 

One adaptation we will have to do in the real proof is that in our previous description of the rounding step, we assumed the weighted hypergraph $\mathcal{H}$ built from a fractional coloring to be uniform. It turns out that this idealistic assumption is not difficult to remove by simply adding a few dummy vertices to make our hypergraph uniform. Concretely, it will turn out to be convenient to enlarge all independent sets (of size at most $m$) by some extra dummy vertices to become of size exactly $m+2$, and doing so while distributing the weights on these extended hyperedges equally and including all dummy vertices equally often. It turns out that doing this in the right way will preserve approximate weighted regularity of the hypergraph also at the dummy vertices, while also keeping the maximum weighted codegrees small. This then enables an application of Theorem~\ref{thm:pippenger}, providing an almost spanning matching in this extended hypergraph which then remains almost spanning after shrinking all the hyperedges by excluding the dummy vertices. 

We are now ready to present the formal proof of Lemma~\ref{lem:rounding}, following the outline above.

\begin{proof}[Proof of Lemma~\ref{lem:rounding}]
Let $n\in\mathbb{N}$, $m\in \mathbb{R}_{\ge 2}$ and $\gamma\in (0,1)$ be such that $n\ge \left(\frac{600(m+2)}{\gamma}\right)^{20m+40}$ and let $G$ be any graph on $n$ vertices with $\alpha(G)\le m$.  Our goal is to show that there is a graph $D$ on the same vertex-set as $G$ such that 
$$\chi(G)\le \chi_f(G\cup D)+\gamma n \text{ and }\Delta(D)\le \left(\frac{600(m+2)}{\gamma}\right)^{20m+41}.$$
Without loss of generality, possibly by replacing $m$ with $\lfloor m\rfloor$, we may assume that $m\in \mathbb{N}$ in the following.

Set $\mu:=\left(\frac{\gamma}{300(m+2)}\right)^{20m+40}$ and in the following let us denote by $\overline{E}$ the set of unordered pairs of distinct vertices of $G$ which are nonadjacent in $G$. Let $(x_I)_{I\in \mathcal{I}(G)}$ as well as $(y_v)_{v\in V(G)}, (z_{uv})_{uv\in\overline{E}}$ be a dual pair of optimal solutions to the modified LPs~\textsc{(P-CODEGREE)} and~\textsc{(D-CODEGREE)} for $G$ and $\mu$, and let $\mathrm{OPT}$ denote their common optimal value. Observe that it follows directly from the constraints of~\textsc{(D-CODEGREE)} that for every $v\in V(G)$ we have $y_v\le 1$ (apply the independent set-constraint to the singleton independent set $\{v\}$). In particular, $\mathrm{OPT}\le \sum_{v\in V(G)}y_v\le n$. 

We next claim that for every vertex $v\in V(G)$, the sum over the variables $z_{uv}$ over incident non-edge pairs is small. Concretely, for every $v\in V(G)$ we claim that
$$\sum_{u\in V(G):uv\in \overline{E}}z_{uv}\le \frac{1}{\mu}.$$
Indeed, suppose that for some $v$ we had $\sum_{u\in V(G):uv\in \overline{E}}z_{uv}> \frac{1}{\mu}$. Then set the value of $y_v$ and all variables $z_{uv}$ with $uv\in \overline{E}$ (i.e. of incident non-edge pairs) to $0$. It is easy to see that this maintains feasibility, since all the independent set constraints for independent sets $I$ including $v$ are equivalent to the previous constraints applied to $I\setminus \{v\}$, and for all independent sets that do not include $v$ the validity of the constraints remains unaffected since none of the involved variables are changed. Since we assumed an optimal dual solution, it follows that this process must not have increased the objective function, which by setting the above-mentioned variables to $0$ has changed by exactly $-y_v+\mu\sum_{u: uv\in \overline{E}}z_{uv}$. Hence, the last expression must be non-positive, which after rearranging implies $\sum_{u\in V(G):uv\in \overline{E}}z_{uv}\le \frac{y_v}{\mu}\le \frac{1}{\mu}$, as desired. 

Let now $D$ be defined as the graph on the same vertex-set as $G$ whose edges are exactly those pairs $uv\in\overline{E}$ for which $z_{uv}\ge \frac{1}{\mu}\left(\frac{\gamma}{600(m+2)}\right)^{20m+41}=\frac{\gamma}{600(m+2) 2^{20m+40}}.$ It then follows directly from our above inequality that at any vertex $v$ the number of incident non-edge pairs $uv$ with $z_{uv}$ at least this threshold can be at most $\left(\frac{600(m+2)}{\gamma}\right)^{20m+41}$, and hence the so-defined graph $D$ satisfies the desired maximum degree bound $\Delta(D)\le \left(\frac{600(m+2)}{\gamma}\right)^{20m+41}$.

Next, consider any independent set $I$ in $G\cup D$. By definition of $D$, this is equivalent to $I$ being an independent set in $G$ such that all pairs $uv\in\overline{E}$ with $\{u,v\}\subseteq I$ satisfy $z_{uv}<\frac{\gamma}{600(m+2) 2^{20m+40}}$. Hence, we find, using the constraint of the program~\textsc{(D-CODEGREE)} at $I$, that
$$\sum_{v\in I}y_v\le 1+\sum_{uv\in\overline{E}: \{u,v\}\subseteq I}z_{uv}\le 1+m^2 \frac{\gamma}{600(m+2)2^{20m+40}}<1+\frac{\gamma}{2},$$ where we used the fact that $|I|\le \alpha(G)\le m$ as well as $m\ge 2$ (the last step is very lossy). This implies that $\left(\left(\frac{1}{1+\frac{\gamma}{2}}\right)y_v\right)_{v\in V(G)}$ forms a feasible solution of the ordinary dual LP~(D) for $G\cup D$, and it follows that
$$\chi_f(G\cup D)\ge \frac{1}{1+\frac{\gamma}{2}}\sum_{v\in V(G)}y_v\ge \left(1-\frac{\gamma}{2}\right)\left(\sum_{v\in V(G)}y_v-\mu\sum_{uv\in\overline{E}}z_{uv}\right)$$
$$=\mathrm{OPT}-\frac{\gamma}{2}\mathrm{OPT}\ge \mathrm{OPT}-\frac{\gamma}{2}n.$$
In particular, it follows that $\sum_{I\in\mathcal{I}(G)}x_I=\mathrm{OPT}\le \chi_f(G\cup D)+\frac{\gamma}{2}n$. Now recall that by feasibility for~\textsc{(P-CODEGREE)}, we also have that 
$$\sum_{I\in \mathcal{I}(G): v\in I}x_I\ge 1$$ for every $v\in V(G)$ as well as
$$\sum_{I\in\mathcal{I}(G):\{u,v\}\subseteq I}x_I\le \mu$$ for every $uv\in \overline{E}$. We claim that we may further assume w.l.o.g. that $\sum_{I\in \mathcal{I}(G): v\in I}x_I= 1$ for every $v\in V(G)$. Indeed, suppose the latter sum equals $1+\varepsilon$ for some $\varepsilon>0$. Then we can find a list of independent sets $J_1,\ldots,J_k$, each containing $v$, and some $\varepsilon_1,\ldots,\varepsilon_k\ge 0$ such that $\sum_{i=1}^k\varepsilon_i=\varepsilon$ and $x_{J_i}\ge \varepsilon_i$ for every $i$. We may then decrease, for every independent set $J_i$, the variable $x_{J_i}$ by $\varepsilon_i$ and increase $x_{J_i\setminus \{v\}}$ by $\varepsilon_i$. This puts the value of the sum $\sum_{I\in \mathcal{I}(G): v\in I}x_I$ down to exactly $1$ while preserving the total sum of all $x_I$ as well as the values of all other sums $\sum_{I\in \mathcal{I}(G): u\in I}x_I$ for all $u\neq v$. Moreover, note that none of the sums $\sum_{I\in\mathcal{I}(G): \{u,v\}\subseteq I}$ for $uv\in\overline{E}$ increase by this modification. Repeating this process for all vertices $v$ in some linear order makes all constraints be satisfied tightly with the sum equal to $1$ while not changing the total sum, as desired. 

In the following, for convenience, let us set $r:=m+2$. We now build an edge--weighted $r$-uniform hypergraph $\mathcal{H}$ as follows. Let $s:=\lceil r\cdot \mathrm{OPT}-n\rceil$, and let us add a set $S$ consisting of exactly $s$ fresh ``dummy'' vertices, disjoint from $V(G)$. So $V(\mathcal{H})=V(G)\cup S$. Then, for every $I\in\mathcal{I}(G)$, we include all the sets $I\cup T$ where $T$ is any subset of $S$ of size exactly $r-|I|$ as hyperedges to $\mathcal{H}$. To check that this is well-defined, we need to ensure that $s\ge r$. However, this is indeed the case, since
$$m\cdot \mathrm{OPT}=m\sum_{I\in \mathcal{I}(G)}x_I\ge \sum_{I\in\mathcal{I}(G)}|I|x_I= \sum_{I\in\mathcal{I}(G)}\sum_{v\in I}x_I=\sum_{v\in V(G)}\sum_{I\in \mathcal{I}(G):v\in I}x_I=n,$$ and hence $s\ge m\cdot \mathrm{OPT}-n+2\cdot\mathrm{OPT}\ge 2\cdot\mathrm{OPT} \ge 2\frac{n}{m}\ge m+2=r$, since our assumption about $n$ clearly implies that $n\ge \frac{m(m+2)}{2}$. This finishes the description of the hypergraph $\mathcal{H}$.

We next define a non-negative real weighting $(x_e')_{e\in E(\mathcal H)}$ of the hyperedges of $\mathcal H$. Intuitively, we simply divide the weight $x_I$ of every original independent set $I$ coming from the optimal solution of the program~\textsc{(P-CODEGREE)} for $G$ equally among all of its extensions $I\cup T$ with $T\in\binom{S}{r-|I|}$ to hyperedges in $\mathcal{H}$. Concretely, given some hyperedge $e=I\cup T$ of $\mathcal H$, where $I\in \mathcal{I}(G)$ and $T\subseteq S$ has size exactly $r-|I|$, we define
$$x_e':=\frac{1}{\binom{s}{r-|I|}}x_I.$$
It is clear from this definition that the total weight of this weighting of $\mathcal{H}$ then equals
$$\sum_{e\in E(\mathcal{H})}x_e'=\sum_{I\in\mathcal{I}(G)}x_I=\mathrm{OPT}.$$

Next, let us consider the (weighted) degrees in $\mathcal{H}$. In the following, with the goal of applying Theorem~\ref{thm:pippenger} in mind, we set $\xi:=\frac{\gamma}{3r}$, noting that $\mu=\left(\frac{\gamma}{300(m+2)}\right)^{20m+40}=\left(\frac{\xi}{100}\right)^{20r}$. Our goal will be to establish that all weighted degrees lie in $[1-\mu,1]$ and that all weighted codegrees are at most~$\mu$. 
For original vertices $v\in V(G)$, it follows directly from our definition of the weights $x_e'$ that
$$\sum_{e\in E(\mathcal H): v\in e}x_e'=\sum_{I\in\mathcal{I}(G): v\in I}x_I=1.$$
Now let $v\in S$ be some fresh ``dummy'' vertex. We then have
$$\sum_{e\in E(\mathcal{H}):v\in e}x_e'=\sum_{I\in \mathcal{I}(G)}\sum_{T\in\binom{S}{r-|I|}: v\in T}\frac{1}{\binom{s}{r-|I|}}x_I$$
$$=\sum_{I\in\mathcal{I}(G)}\frac{r-|I|}{s}x_I=\frac{r}{s}\mathrm{OPT}-\frac{1}{s}\sum_{I\in\mathcal{I}(G)}|I|x_I$$
$$=\frac{r}{s}\mathrm{OPT}-\frac{n}{s}=\frac{r\cdot\mathrm{OPT}-n}{s}.$$

By definition of $s$, we have $r\cdot\mathrm{OPT}-n\le s<r\mathrm{OPT}-n+1$, implying that the above expression for the weighted degree of $v$ lies in $[1-\frac{1}{s},1]\subseteq [1-\mu,1]$ since $s\ge 2\frac{n}{m}\ge \frac{n}{m}$ as established previously, and since our assumed lower bound on $n$ and $m\ge 2$ easily imply that $\frac{n}{m}\ge \frac{1}{m}\left(\frac{600(m+2)}{\gamma}\right)^{20m+40}\ge \left(\frac{300(m+2)}{\gamma}\right)^{20m+40}=\frac{1}{\mu}$. Hence, the weighted hypergraph $\mathcal{H}$ is indeed approximately regular.

Next, let us give upper bounds on the (weighted) codegrees. There are three potential types of pairs of distinct vertices to consider, namely when both vertices are original vertices, one vertex is original and the other is a dummy, and finally if both vertices are dummies. 

First, suppose that $u$ and $v$ are distinct original vertices (in $V(G)$). We then have, by the definition of our weights $x_e'$, that 
$$\sum_{e\in E(\mathcal H):\{u,v\}\subseteq e}x_e'=\sum_{I\in\mathcal{I}(G): \{u,v\}\subseteq I}x_I\le \mu$$ by the feasibility of $(x_I)_{I\in\mathcal{I}(G)}$ for~\textsc{(P-CODEGREE)}. 

Next, let $v\in V(G)$ be an original vertex and $u\in S$ a dummy vertex. We then obtain, by definition of the $x_e'$ that
$$\sum_{e\in E(\mathcal H):\{u,v\}\subseteq e}x_e'=\sum_{I\in\mathcal{I}(G): v\in I}\frac{\binom{s-1}{r-|I|-1}}{\binom{s}{r-|I|}}x_I$$
$$=\frac{1}{s}\sum_{I\in\mathcal{I}(G): v\in I}(r-|I|)x_I\le \frac{r}{s}\sum_{I\in\mathcal{I}(G): v\in I}x_I$$ $$=\frac{r}{s}\le \frac{m(m+2)}{n}\le m(m+2)\left(\frac{\gamma}{600(m+2)}\right)^{20m+40}\le \left(\frac{\gamma}{300(m+2)}\right)^{20m+40}=\mu,$$ as desired.
Finally, suppose $u,v\in S$ are distinct dummies. Then 
$$\sum_{e\in E(\mathcal H):\{u,v\}\subseteq e}x_e'=\sum_{I\in \mathcal{I}(G)}\frac{\binom{s-2}{r-|I|-2}}{\binom{s}{r-|I|}}x_I$$
$$=\frac{1}{s(s-1)}\sum_{I\in\mathcal{I}(G)}(r-|I|)(r-|I|-1)x_I$$
$$\le \frac{r^2}{s(s-1)}\cdot \mathrm{OPT}\le \frac{2(m+2)^2}{s^2}\cdot \mathrm{OPT}$$ $$\le \frac{2(m+2)^2m^2}{n^2}\cdot n=\frac{2(m+2)^2m^2}{n}$$ $$\le 2(m+2)^2m^2\left(\frac{\gamma}{600(m+2)}\right)^{20m+40}\le \left(\frac{\gamma}{300(m+2)}\right)^{20m+40}=\mu,$$ again the desired bound.

All in all, it follows that $\mathcal{H}$ equipped with $(x_e')_{e\in E(\mathcal H)}$ is an $r$-uniform edge-weighted hypergraph where all weighted degrees lie in $[1-\mu,1]$ and all weighted codegrees are bounded by $\mu$. We may thus apply Theorem~\ref{thm:pippenger} to $\mathcal H$ with parameters $\xi$ and $\mu$ to find that there exists a matching $M$ in $\mathcal H$ covering all but at most $\xi\v(\mathcal H)=\xi(n+s)$ vertices of $\mathcal H$. Let $I_1,\ldots, I_{|M|}$ be an enumeration of all the intersections of the hyperedges in $M$ with the set $V(G)$ of original vertices. Then, by definition of $\mathcal H$, we have that $I_1,\ldots, I_{|M|}$ are disjoint independent sets in $G$ covering all but at most $\xi(n+s)$ vertices of $G$. Giving each independent set one color and assigning each uncovered vertex a unique color, we now obtain that
$$\chi(G)\le |M|+\xi(n+s)\le \frac{\v(\mathcal H)}{r}+\xi(n+s)=\frac{n+s}{r}+\xi(n+s).$$ Noting that $n+s<r\cdot \mathrm{OPT}+1$ and recalling $\mathrm{OPT}\le n$ and $\xi=\frac{\gamma}{3r}$, we find that
$$\chi(G)\le \mathrm{OPT}+\frac{1}{r}+\xi r (\mathrm{OPT}+1)\le \mathrm{OPT}+\frac{1}{r}+\frac{\gamma}{3}(n+1)<\mathrm{OPT}+\frac{\gamma}{2}n,$$ where we used that $n\gg \frac{1}{\gamma}$ in the last step.

Putting this together with our previous estimate that $\mathrm{OPT}\le \chi_f(G\cup D)+\frac{\gamma}{2}n$, we obtained the desired bound
$$\chi(G)\le \chi_f(G\cup D)+\frac{\gamma}{2}n+\frac{\gamma}{2}n=\chi_f(G\cup D)+\gamma n$$ on the chromatic number of $G$. This concludes the proof of the lemma.
\end{proof}

\subsection{Bounding the fractional chromatic number after augmentation}\label{sec:rapha-bounding}

In this section, we prove Lemma~\ref{lem:bounding}. Again, we first give a high-level overview of the idea of the proof before delving into the details of the formal proof. Recall that our goal in the lemma is to prove an upper bound on the fractional chromatic number of a union $G\cup D$ where $G$ satisfies $h(G)<t$ and has small independence number $\alpha(G)\le m$, while $D$ is a graph on the same vertex set that has bounded maximum degree (by some parameter $d$). As previously explained, we want to avoid the naive product estimate since the obtained bound is much too bad for our purpose, and instead want to prove an additive error bound as stated in the lemma: A linear term in $t$ plus another term only depending on $d$ and $m$, growing like $d^{\Theta(m)}$. To do so, we will apply the resolution of Linear Hadwiger's conjecture for fractional coloring due to Reed and Seymour~\cite{RS}, stated previously in Theorem~\ref{thm:RS} and restated for convenience below, not only to the graph $G$ itself, but (crucially!) also to various minors of $G$ obtained by contracting matchings in an induced subgraph.

\begin{theorem}[cf.~\cite{RS}, Theorem~1.3]\label{thm:fractional}
Every graph $F$ satisfies $\chi_f(F)\le 2\cdot h(F)$.
\end{theorem}

In the proof of Lemma~\ref{lem:bounding}, it turns out to be most convenient to look at things from a dual perspective again: From the formulation of the fractional chromatic number of $G\cup D$ as the optimal value of the dual LP~(D), one can see that $\chi_f(G\cup D)$ equals the smallest real number $x$ such that for every non-negative vertex-weighting $w:V(G)\rightarrow \mathbb{R}_{\ge 0}$, there exists an independent set $I$ in $G\cup D$ of weight at least a $\frac{1}{x}$-fraction of the total weight. Hence, in the proof we will assume that we are given such a weighting and aim to find such a ``heavy'' independent set in $G\cup D$. Note that by Theorem~\ref{thm:fractional}, this would be easy (with $x:=2t$) if we were just looking to find a heavy independent set $I$ in $G$. But there could (and will) be additional edges from $D$ with both endpoints in $I$, making our life hard. In an ideal scenario where the induced subgraph $D[I]$ is sparse, such as a forest, say; we would nevertheless be happy, since we could then move to the heavier half of a bipartition of said forest to find a subset of $I$ which is independent in $G\cup D$ that still keeps a $\frac{1}{4t}$-fraction of the total weight, which would be much to our liking (recall that we aim to prove a bound of the form $4t+d^{\Theta(m)}$ on $\chi_f(G\cup D)$). Note that since $\alpha(G)\le m$, the induced subgraph $D[I]$ would actually be fairly small (of order at most $m$), and so all we would need to do to make this work is to find an independent set $I$ which does not interfere much with the short cycles in $D$.

However, this is obviously wishful thinking, as from blindly applying Theorem~\ref{thm:fractional} we have almost no control where in $D$ the heavy independent set $I$ in $G$ lands, and furthermore, even though the graph $D$ is fairly sparse (has bounded maximum degree), it can still host a lot of short cycles all over the place. 

The first key idea of the proof of Lemma~\ref{lem:bounding} is to observe a significantly weaker sufficient condition for extracting a heavy independent set from a heavy subset $U$ of vertices, which does not require $(G\cup D)[U]$ to avoid cycles entirely but only cycles of a certain type: 

Imagine we can find a heavy subset $U$ of $V(G)=V(D)$ such that no cycle in $(G\cup D)[U]$ uses an edge of $G$, i.e. all the $G$-edges are bridges of the induced subgraph. Suppose further that the edges in $G[U]$ form a matching $N$ which spans $U$. It is then fairly easy to see, based on the component structure of $(G\cup D)[U]-N$, that one can find an independent transversal of the matching $N$, i.e. an independent set $I$ in $G\cup D$ which from each matching edge in $N$ takes exactly one endpoint (this is formalized in Observation~\ref{obs:bridges} below). Then, as long as for every edge in $N$ the weights of the endpoints are quite similar to each other, the independent set $I$ obtained in this way will have a total weight close to half of the total weight of $U$, and hence still be heavy. 

Of course, this still sounds a lot like wishful thinking. Before continuing to discuss how we will actually arrive at a subset $U$ with these magical properties, let us first give the simple statement and proof of the aforementioned observation.

\begin{observation}\label{obs:bridges}
Let $H$ be a graph and let $N$ be a perfect matching of $H$, such that no cycle in $H$ uses an edge of $N$. Then there exists an independent set $I$ in $H$ which includes exactly one endpoint of every matching edge.
\end{observation}
\begin{proof}
Let $H'$ be the multi-graph obtained from $H$ by contracting all edges in $E(H)\setminus N$ (here, exceptionally, if this gave rise to loops or multiple edges, we would keep them, in particular, all edges in $N$ are retained as edges in $H'$). 

However, by our assumption, $H'$ must in fact be a simple graph and a forest, since any loop, pair of multi-edges or cycle in $H'$ would give rise to a cycle in $H$ including one or several edges of $N$. 

Now, fix an orientation of the forest $H'$ of maximum out-degree at most $1$ (e.g., by selecting a root for every component and directing the edges in each component towards its root). Note that since the edges of $H'$ correspond exactly to the edges in the perfect matching $N$, this also induces an orientation of the matching edges in $N$. Let $I\subseteq V(H)$ be the set of starting points of the matching edges in $N$ under this orientation. Then $I$ is independent in $H$: If two distinct vertices $u$ and $v$ of $I$ would be adjacent in $H$ (necessarily via an edge in $E(H)\setminus N$), then the two distinct edges of $N$ starting in $u$ and $v$ would correspond in $H'$ to two distinct outgoing edges of the same vertex obtained by contracting the edge $uv$, contradicting the fact that we chose an orientation of $H'$ of maximum out-degree at most $1$. Since $I$ has the required properties, we may conclude the proof.
\end{proof}

Let us now discuss how we can hope to find a heavy subset $U$ with the desired properties, which we restate here in a qualitative manner for convenience:

\begin{itemize}
    \item $G[U]$ is a matching $N$ that spans all vertices in $U$,
    \item $U$ takes a decent fraction of the total weight of our initial vertex-weighting of $G\cup D$,
    \item For a typical matching edge $uv\in N$, the weights of $u$ and $v$ are not too dissimilar, and
    \item No cycle in $(G\cup D)[U]$ uses an edge in $G$ (i.e., of $N$).
\end{itemize}

Let us forget about the third and fourth conditions for now and let us first see how to get a heavy subset $U$ such that $G[U]$ is a matching that spans $U$. A natural idea to find this heavy induced matching is as follows: Find a much bigger matching $M$ in $G$ whose endpoints ideally take up most of the total vertex-weight in $G$, i.e. such that the weight of uncovered vertices is not too large. Then, instead of applying Theorem~\ref{thm:fractional} to all of $G$ to extract a heavy independent set, apply it to the \emph{minor} $G':=G[V(M)]/M$ of $G$ obtained from the subgraph induced by the endpoints of $M$ by contracting the matching, and equip it with the natural associated weighting, where the weight of a contracted matching edge corresponds to the sum of weights of its endpoints, say. 

Since $h(G')\le h(G)<t$, Theorem~\ref{thm:fractional} then guarantees a heavy independent set in $G'$ taking at least a $\frac{1}{2t}$-fraction of the total weight, which upon uncontracting yields a submatching $N$ of $M$ which is \emph{induced in $G$} and takes at least a $\frac{1}{2t}$ fraction of the weight of the whole matching $M$, which we assumed to be not much smaller than the weight of the whole initial weighting of $G\cup D$. Hence, the set $U$ of endpoints of $N$ also takes a decent fraction of the total weight, and $G[U]$ indeed is a matching that spans $U$. Additionally note that the set $U$ obtained in this way will be fairly small: taking one endpoint of every edge in $N$ yields an independent set in $G$, such that we must have $|N|\le \alpha(G)\le m$ and thus $|U|\le 2m$. In particular, all cycles in $(G\cup D)[U]$ will have length at most~$2m$.

Summarizing, we now have a suitable mechanism for satisfying the first two points. This mechanism should make it plausible that we can meet the requirements of all four points above provided we can find a matching $M$ in $G$ such that:
\begin{itemize}
    \item $M$ does not leave a lot of vertex-weight from the given weighting of $G\cup D$ uncovered (will ensure the first two bullets above),
    \item the sum over all edges $uv$ in $M$ of the absolute difference between the weight of $u$ and the weight of $v$ is not too large (responsible for the third bullet above), and
    \item no cycle of length at most $2m$ in $M\cup D$ uses a matching edge from $M$ (will ensure the fourth bullet above).
\end{itemize}

The fact that it is indeed always possible to find a matching with these properties under our hypotheses is quite surprising and relies on a simple greedy algorithm with a clever analysis, which makes crucial use of the bounded degree assumption on $D$ to control the number of vertices reachable from a fixed vertex along short paths in $D$, which is relevant to avoid short cycles in $M\cup D$ passing through matching edges while building the matching $M$. This statement can be cleanly isolated by the following auxiliary result.

\begin{lemma}\label{lem:magicalmatching}
Let $m, d\in \mathbb{R}_{\ge 1}$, $r\in\mathbb{N}$ and let $G$ and $D$ be graphs on the same vertex-set such that $\alpha(G)\le m$ and $\Delta(D)\le d$. Then for every vertex-weighting $w:V(G)=V(D)\rightarrow [0,1]$ there exists a matching $M$ in $G$ with $M\cap E(D)=\emptyset$ such that no cycle of length at most $r+1$ in $M\cup D$ uses a matching edge of $M$, and furthermore, denoting by $R$ the set of vertices in $G$ that are not covered by $M$, we have
$$\sum_{v\in R}w(v)+\sum_{uv\in M}|w(u)-w(v)|\le m\left(1+\sum_{j=0}^r (d+1)^j\right).$$
\end{lemma}

Before we go into the details of the proof of this important lemma, let us first summarize our discussion so far by demonstrating formally that Lemma~\ref{lem:magicalmatching} is sufficient to establish our second key lemma, Lemma~\ref{lem:bounding}.

\begin{proof}[Proof of Lemma~\ref{lem:bounding}, assuming Lemma~\ref{lem:magicalmatching}]
Let $m,d,t\in \mathbb{R}_{\ge 1}$, let $G$ be a graph with $h(G)<t$ and $\alpha(G)\le m$ and $D$ a graph on the same vertex-set as $G$ such that $\Delta(D)\le d$. Our goal is to prove that $\chi_f(G\cup D)\le 4t+(d+1)^{4m}$. Observe that possibly after replacing $m$ by $\lfloor m\rfloor$ and $d$ by $\lfloor d\rfloor$ in the following, we may w.l.o.g. assume that $m,d\in\mathbb{N}$. 

Let $(y_v)_{v\in V(G)}$ be an optimal solution of the dual LP~(D) for the fractional chromatic number of $G\cup D$. We thus have that the $y_v$ are non-negative, as well as $\sum_{v\in V(G)}y_v=\chi_f(G\cup D)$ and $\sum_{v\in I}y_v\le 1$ for every $I\in \mathcal{I}(G\cup D)$. In particular, $y_v\le 1$ for every $v\in V(G)$ by applying the latter inequality with $I=\{v\}$. 

Let $w:V(G)=V(D)\rightarrow [0,1]$ be the weighting defined by $w(v):=y_v$ for every $v\in V(G)$. Then the total weight of $w$ equals $\chi_f(G\cup D)$. Now, let us apply Lemma~\ref{lem:magicalmatching} to the graphs $G$, $D$ with parameters $m$, $d$ and $r:=2m-1$ and the vertex-weighting $w$. The lemma thus yields a matching $M$ in $G$, edge-disjoint from $D$, such that no cycle of length at most $r+1=2m$ in $M\cup D$ uses an edge in $M$, and such that
$$\sum_{v\in R}w(v)+\sum_{uv\in M}|w(u)-w(v)|\le m\left(1+\sum_{j=0}^{2m-1} (d+1)^j\right)<m(1+(d+1)^{2m}),$$
where $R$ denotes the set of vertices in $V(G)=V(D)$ not covered by $M$.

Let now $G':=(G-R)/M$ be the minor of $G$ obtained by first deleting $R$ and then contracting the edges in $M$. Clearly, the vertices in $G'$ are in one-to-one correspondence with the edges of $M$. For a given edge $m\in M$, let us denote by $\mathrm{ctr}(m)$ the corresponding vertex of $G'$. 

Let us now define a non-negative weighting $w':V(G')\rightarrow \mathbb{R}_{\ge 0}$ of $G'$ by setting $w'(\mathrm{ctr}(uv)):=\min\{w(u),w(v)\}$ for every $uv\in M$.\footnote{The alert reader may notice that this definition for a weight function for $G'$ is slightly different from the weight function where we take the sum of the weights of the endpoints of a matching edge, which we previously indicated in the high-level overview. It turns out that this is the better and more convenient choice for the ensuing formal argument, but for the proof sketch we felt it more appropriate to stick with the sum-weighting to convey the essential intuition--observe that by our choice of $M$, the difference $|w(u)-w(v)|$ will be typically quite small, making $w(u)+w(v)$ typically very similar to $2\min\{w(u),w(v)\}$, and the constant factor of two is irrelevant here since the following arguments are scale-invariant.}
Since $G'$ is a minor of $G$, we have $h(G')\le h(G)<t$, and hence $\chi_f(G')\le 2t$ by Theorem~\ref{thm:fractional}. It follows (as explained before, and can be verified using the dual LP~(D)) that there exists an independent set $J$ in $G'$ such that $$\sum_{x\in J}w'(x)\ge\frac{1}{\chi_f(G')}\sum_{x\in V(G')}w'(x)\ge \frac{1}{2t}\sum_{uv\in M}\min\{w(u),w(v)\}$$
$$=\frac{1}{2t}\sum_{uv\in M}\frac{w(u)+w(v)-|w(u)-w(v)|}{2}$$
$$=\frac{1}{4t}\left(\sum_{x\in V(G)\setminus R}w(x)-\sum_{uv\in M}|w(u)-w(v)|\right)$$
$$=\frac{1}{4t}\left(\chi_f(G\cup D)-\left(\sum_{v\in R}w(v)+\sum_{uv\in M}|w(u)-w(v)|\right)\right)$$
$$\ge \frac{1}{4t}\left(\chi_f(G\cup D)-m(1+(d+1)^{2m})\right).$$
Let now $N:=\{m\in M|\mathrm{ctr}(m)\in J\}$ be the submatching of $M$ corresponding to $J$, and let $U$ denote the set of vertices covered by $N$. Then the induced subgraph $G[U]$ consists exactly of the spanning matching $N$ and no further edges. In particular, selecting any one endpoint per edge yields an independent set in $G$, implying that $|N|\le \alpha(G)\le m$ and thus $|U|=2|N|\le 2m$. Consequently, every cycle in $H:=(G\cup D)[U]$ is a cycle of length at most $2m$ in $M\cup D$, which by the guarantee about $M$ coming from our application of Lemma~\ref{lem:magicalmatching} cannot use any edge in $M$ (and thus no edge in $N$). We have thus established that $N$ is a perfect matching in the graph $H$ such that no cycle in $H$ uses an edge from the matching. By Observation~\ref{obs:bridges}, it now follows that there exists an independent set $I$ in $H$ (and thus in $G\cup D$) which consists of exactly one endpoint from every matching edge in $N$. In particular, it follows that
$$1\ge \sum_{v\in I}y_v=\sum_{v\in I}w(v)\ge \sum_{uv\in N}\min\{w(u),w(v)\}$$
$$=\sum_{uv\in N}w'(\mathrm{ctr}(uv))=\sum_{x\in J}w'(x)$$
$$\ge \frac{1}{4t}\left(\chi_f(G\cup D)-m(1+(d+1)^{2m})\right).$$
Rearranging now yields that
$$\chi_f(G\cup D)\le 4t+m(1+(d+1)^{2m})\le 4t+(d+1)^{4m},$$ where the last inequality is easily checked to hold for every $d,m\ge 1$. This establishes the claim of Lemma~\ref{lem:bounding} and concludes its proof.
\end{proof}

Finally, we supply the proof of Lemma~\ref{lem:magicalmatching}, which is the last missing piece towards completing the proof of Lemma~\ref{lem:bounding}. As previously mentioned, the matching $M$ as in the lemma is built using a very simple greedy algorithm, however, the analysis proving that it has the properties required by the lemma is clever and elegant. On a high level, the algorithm processes vertices in decreasing order of their weight and tries to match a given vertex to as heavy a neighbor of smaller weight as possible, while not creating the short cycles forbidden by the lemma. Intuitively, we hope to ensure that by matching vertices to as heavy neighbors below their own weight as possible we obtain a matching where typically, matched pairs have similar weight. In general, this can go very wrong, however in our case it is plausible, since

\begin{enumerate} 
\item[(a)] the assumption that $G$ has small independence number gives typical vertices a lot of flexibility in terms of potential matching partners to choose from, and
\item[(b)] the bounded-degree assumption on the graph $D$ limits the effect of the short cycle-side constraint on the process. 
\end{enumerate}

\begin{proof}[Proof of Lemma~\ref{lem:magicalmatching}]
Let parameters $m,d,r$, graphs $G,D$ and a weighting $w:V(G)=V(D)\rightarrow [0,1]$ be given satisfying the properties stated in the lemma. We build the matching $M$ using a simple greedy algorithm which processes vertices one by one and decides for each of them where to match them (or whether to leave them uncovered), as described by the pseudocode below.

During the algorithm, we keep track of a matching $M$ in the graph $G$ (which is initially empty and at the end of the algorithm will be the desired matching), a subset $P$ of $V(G)=V(D)$ consisting of \emph{processed} vertices, and its complement, the set $U:=V(G)\setminus P=V(D)\setminus P$ of currently \emph{unprocessed} vertices. The algorithm starts from an empty set of processed vertices and finishes when all vertices have been processed. We maintain the invariant that the current matching $M$ is contained in the subgraph $G[P]$ induced by the processed vertices, and that at every time no cycle of length at most $r+1$ in the graph $M\cup D$ defined with the edges in the current matching $M$ and $D$ uses any edge of $M$. We also keep track of the set $R:=P\setminus V(M)$ of processed vertices left uncovered by $M$. In the following, we abbreviate $V:=V(G)=V(D)$, and for a vertex $v\in V$ and a graph $H$ on vertex-set $V$ we will denote by $B_r^H(v)$ the ball of radius $r$ around $v$ in $H$, i.e. the set of all vertices in $H$ at distance at most $r$ from $v$.

\newpage
\begingroup
\captionof{algorithm}{\textbf{Weight-balanced matching}}
\label{alg:magicmatching}
\begin{algorithmic}[1]
\State \textbf{initialize} $M\gets \emptyset$, $P\gets \emptyset$, $U\gets V$, $R\gets \emptyset$. 
\While{$U\neq \emptyset$}
\State $H\gets M\cup D$
\State Pick $v$ in $\mathrm{argmax}\{w(x)|x\in U\}$
\If{$(N_G(v)\cap U)\setminus B_r^H(v)=\emptyset$}
\State $R\gets R\cup \{v\}$,  $P\gets P\cup \{v\}$, $U\gets U\setminus \{v\}$
\Else
\State Pick $u$ in $\mathrm{argmax}\{w(x)|x\in (N_G(v)\cap U)\setminus B_r^H(v)\}$
\State $M\gets M\cup\{uv\}$,  $P\gets P\cup \{u,v\}$, $U\gets U\setminus \{u,v\}$
\EndIf
\EndWhile

\medskip

\noindent\Return $M$ and $R$
\end{algorithmic}
\endgroup

\medskip

Let us also briefly describe this simple algorithm in words: While there still is a vertex left to process, we pick a maximum-weight unprocessed vertex $v$, which we seek to match with another unprocessed vertex $u$ of high weight. However, it may be the case that all the neighbors of $v$ in $G$ are already processed or that all of them are very close to $v$ in the current graph $H=M\cup D$ consisting of the already built matching edges and $D$ (formally, $(N_G(v)\cap U)\setminus B_r^H(v)=\emptyset$). In this case we wish to avoid matching $v$ at all, since matching $v$ to a neighbor $u$ very close to it in $D\cup M$ would either use an edge already in $D$ or create a short cycle (of length at most $r+1$) through the new matching edge $uv$, both of which we cannot allow if we want to meet the requirements of the lemma. Hence, in this case we simply declare $v$ as processed and move it to the unmatched set $R$, and continue.

However, if this is not the case, i.e., if $(N_G(v)\cap U)\setminus B_r^H(v)\neq\emptyset$ and there are still ``safe'' unprocessed neighbors $u$ of $v$ in $G$, then we pick among them some vertex $u$ of maximum weight and match $v$ to $u$, declare $u$ and $v$ as processed and move on. 

We repeat this process until all vertices have been processed, which clearly happens after finitely many steps since at each step we process one or two new vertices. It should be clear from the procedure that the final returned $M$ is a matching in $G$ and that $R$ is exactly the set of vertices in $G$ left uncovered by $M$. Furthermore, our choice of the neighbor $u$ to match $v$ to ensures that at any point in the process, when we add a new matching edge $uv$ to the current matching $M$ and update $M$ to $M':=M\cup\{uv\}$, then $u$ and $v$ had distance at least $r+1$ in the graph $M\cup D$, and hence no new cycle of length at most $r+1$ through the new edge $uv$ can be created in the new graph $M'\cup D$. 

Since $M$ is initially empty, this immediately shows that at every stage of the process no cycle of length at most $r+1$ in $M\cup D$ uses an edge of $M$, and hence this is also the case for the final returned matching $M$. 

Hence, we already established that the output $M$ of the above algorithm is a matching in $G$ satisfying the first condition required by the lemma. To conclude the proof, it thus remains to show that $M$ and $R$ satisfy the inequality
$$\sum_{v\in R}w(v)+\sum_{uv\in M}|w(u)-w(v)|\le m\left(1+\sum_{j=0}^r (d+1)^j\right)$$ as claimed in the lemma.

The crucial step to show this is to reinterpret the left hand side of the inequality above as the expectation of a certain random variable. For this, imagine we randomly and uniformly draw a real threshold $t\in [0,1]$ and define a subset $X_t$ of $V$ by including in $X_t$ all unmatched vertices $v\in R$ with weight above the threshold, i.e., $w(v)\ge t$, as well as all vertices $v$ where $uv\in M$ is a matching edge such that $w(u)<t<w(v)$, i.e. when the random threshold $t$ separates the weights of the endpoints of a matching edge, then we include the heavier of the two endpoints into $X_t$ (and none of the endpoints if they have equal weight). 

We first observe that the left hand side of the desired inequality equals the expected size of the random set $X_t$ defined in this way. Indeed, for every fixed vertex $v\in R$, by definition the probability it is included in $X_t$ is exactly the probability that $t$ lies in $[0,w(v)]$, so equal to $w(v)$, and for every vertex $v$ which is an endpoint of a matching edge $uv\in M$ with $w(v)>w(u)$, the probability it lies in $X_t$ is exactly the length of the open interval $(w(u),w(v))$, and so equals  $w(v)-w(u)=|w(u)-w(v)|$. Finally, if $u,v$ are endpoints of a matching edge $uv\in M$ with $w(u)=w(v)$, then none of them is included into $X_t$, so the inclusion happens with probability $0=|w(u)-w(v)|$.

It follows from summing these probabilities (using linearity of expectation), that we indeed have
$$\mathbb{E}|X_t|=\sum_{v\in R}w(v)+\sum_{uv\in M}|w(u)-w(v)|.$$
With this reformulation of the left hand side of the desired inequality at hand, to conclude the proof of the lemma, it suffices to show that the expected size of the set $X_t$ is at most $m\left(1+\sum_{j=0}^{r}(d+1)^j\right)$. We will do so by showing that actually $|X_t|\le m\left(1+\sum_{j=0}^{r}(d+1)^j\right)$ for \emph{every} choice of $t$, which will certainly also imply the desired bound on the expectation.

So let $t\in [0,1]$ be given arbitrarily, and let us consider the set $X_t$. Our strategy for upper-bounding $|X_t|$ will be to show that the subgraph $G[X_t]$ of $G$ induced by $X_t$, by virtue of our construction of the matching, has small \emph{degeneracy} (here we will crucially use the bounded-degree assumption on $D$ to control ball sizes). We can then conclude that $G[X_t]$ also has small chromatic number, and since each independent set is small in $G$ (as $\alpha(G)\le m$ by assumption), it will also have to have small size, as desired.

To witness this low degeneracy of $G[X_t]$, let us order the vertices in $X_t$ as $x_1,\ldots,x_k$ for some $k\ge 0$, exactly in the order they were processed by our algorithm (i.e., $x_1$ was added to $P$ first, then $x_2$, then $x_3$, and so on).

We now claim that for every $i\in [k]$, all neighbors of $x_i$ in $G$ among later vertices in $X_t$ lie in the ball of radius $r$ around $x_i$ in $M\cup D$. Concretely, 
$$N_G(x_i)\cap \{x_{i+1},\ldots,x_k\}\subseteq B_r^{M\cup D}(x_i).$$
Indeed, suppose towards a contradiction that there exists some $j\in \{i+1,\ldots,k\}$ such that $x_j$ is a neighbor of $x_i$ in $G$ that does not lie in $B_r^{M\cup D}(x_i)$. 

In the following, let us denote by $U$ the set of unprocessed vertices in Algorithm~\ref{alg:magicmatching} at the beginning of the iteration of the while loop in which $x_i$ is processed. We then claim that $x_i$ was the vertex in $\mathrm{argmax}\{w(x)|x\in U\}$ selected in line $4$ of the algorithm.

Indeed, it can be readily seen from how the algorithm processes vertices that the only alternative would be if $x_i$ was the matching partner in $M$ of the vertex $v\in \mathrm{argmax}\{w(x)|x\in U\}$ selected in line 4 of the algorithm. But this would mean that the weight of $x_i$ is at most the weight of its matching partner in $M$, which contradicts the fact that $x_i\in X_t$. Hence, indeed, $x_i$ was a heaviest vertex selected in line 4 of the algorithm in the iteration of the while loop in which it was processed.

Let us now come back to considering $x_j$ and to deriving the desired contradiction.

Let $M'\subseteq M$ denote the matching at the beginning of the iteration of the while loop in which $x_i$ was processed. Then, by definition of our ordering, $x_j$ was still unprocessed at this time, so $x_j\in U$, and since $B_r^{M'\cup D}(x_i)\subseteq B_r^{M\cup D}(x_i)$, we must have had $x_j\in (N_G(x_i)\cap U)\setminus B_r^H(x_i)$ with $H:=M'\cup D$ at this point of the process. In particular, $(N_G(x_i)\cap U)\setminus B_r^H(x_i)\neq \emptyset$, which means that in this iteration of the loop we entered the else-branch. In particular, it follows that $x_i$ was matched to a vertex $u\in \mathrm{argmax}\{w(x)|x\in (N_G(x_i)\cap U)\setminus B_r^H(x_i)\}$ at this point of the process.

This implies that $w(u)\ge w(x_j)$, where $u$ is the unique matching partner of $x_i$ in $M$. However, since $x_i\in X_t$, we have $w(u)<t<w(x_i)$, and since $x_j\in X_t$, we have $w(x_j)\ge t$. This implies $w(x_j)>w(u)$, contrary to what we just stated. This is the desired contradiction, showing that the claimed inclusion
$$N_G(x_i)\cap \{x_{i+1},\ldots,x_k\}\subseteq B_r^{M\cup D}(x_i)$$
indeed holds for every $i\in [k]$. Since $D$ has maximum degree at most $d$ by assumption and since $M$ is a matching, $M\cup D$ has maximum degree at most $d+1$, and so the $r$-ball in the above inclusion can contain at most $\sum_{j=0}^{r}(d+1)^j$ vertices. Hence, the ordering $x_1,\ldots,x_k$ of the vertices of $G[X_t]$ witnesses that $G[X_t]$ is $\sum_{j=0}^{r}(d+1)^j$-degenerate. Since the chromatic number of any graph is bounded above by its degeneracy plus $1$, it follows that
$$\chi(G[X_t])\le 1+\sum_{j=0}^{r}(d+1)^j.$$
Since $\alpha(G[X_t])\le \alpha(G)\le m$ by assumption, it now follows that
$$|X_t|\le \alpha(G[X_t])\chi(G[X_t])\le m\left(1+\sum_{j=0}^{r}(d+1)^j\right).$$
This establishes the desired upper bound on $|X_t|$ for every $t\in [0,1]$, and thus on the expectation of the random variable $|X_t|$. As previously explained, this proves the desired inequality
$$\sum_{v\in R}w(v)+\sum_{uv\in M}|w(u)-w(v)|\le m\left(1+\sum_{j=0}^r (d+1)^j\right),$$
completing the proof of Lemma~\ref{lem:magicalmatching}.
\end{proof}

\section{Proof of the exponent improvement}\label{sec:bootstrap}

In this section we prove Theorem~\ref{thm:bootstrap}. That is we show how to bootstrap a linear bound on the chromatic number for graphs with $\had(G) < t$ expanding the range to which it applies (at the cost of increasing the implicit constant) from   $\v(G) \leq t (\log t)^{\alpha}$ to $\v(G) \leq t (\log t)^{4\alpha/3}$.  Once again we start with a detailed proof overview. 

\subsection{Proof overview}\label{sec:overviewBootstrap}

The most natural way of accomplishing the needed bootstrap is to bound $\chi(G)$ by a linear function of $\chi(G')$ for a minor $G'$ of $G$ such that $\v(G') \leq \frac{\v(G)}{(\log t)^{\alpha/3}}.$ 

Such a linear bound can be guaranteed if every induced  connected subgraph contracted to obtain $G'$ from $G$ has bounded chromatic number. Let us record a simple lemma that makes this observation explicit for later use.

\begin{lemma}\label{lem:contraction} Let $G$ be a graph and let $s \geq 1$ be an integer. Let $(V_1,\ldots,V_m)$ be a partition of $V(G)$ such that $G[V_i]$ is connected and $\chi(G[V_i]) \leq s$ for every $i \in [m]$. Let $G'$ be a minor of $G$ obtained by contracting each $V_i$ to a single vertex. Then $$ \chi(G) \leq s\,\chi(G') .$$
\end{lemma}
\begin{proof} Let $v_i \in V(G')$ be obtained by contraction of $V_i$. Let $s' = \chi (G')$.
     Let $c': V(G') \to [s']$ be a proper coloring of $G'$ and let $c'': V(G) \to [s]$ be an (improper) $s$-coloring of $G$ which is proper when restricted to $G[V_i]$ for each $i$. Let  $c: V(G) \to [s'] \times [s]$ be defined by $c(v) = (c'(v_i),c''(v))$ for $v \in V_i$. We claim that $c$ is a proper coloring of $G$.

     Indeed, if $c(v) = c(v')$ for a pair of distinct $v,v' \in V(G)$ such that $v \in V_i,v' \in V_j$ then either \begin{itemize} \item $i=j$, in which case $v$ and $v'$ are non-adjacent, as $c''(v)=c''(v')$ and $c''$ properly colors $G[V_i]$, or \item $i \neq j$, in which case $v$ and $v'$ are non-adjacent, as $c'(v_i)= c'(v_j)$ and  so $v_i$ and $v_j$ are non-adjacent in $G'$, implying that no vertex of $V_i$ is adjacent to any vertex of $V_j$.\end{itemize}
\end{proof}

We will only apply Lemma~\ref{lem:contraction} with $s=2$, i.e. we will only contract bipartite subgraphs of $G$.

\vskip 5pt
Returning to the overview, we'd like to find a minor $G'$ of $G$ obtained by contracting induced  connected bipartite subgraphs, so that  $\v(G') \leq \frac{\v(G)}{k}$, where $k = (\log t)^{\alpha/3}$. The upper bound on $\v(G')$ can be naturally ensured by requiring that every vertex of $G'$ corresponds to a subgraph of $G$ of size $k$. Thus it suffices to divide the vertex set of $G$ into subsets inducing such large bipartite subgraphs.

This is not always possible, but we can try to divide as much as we can. Repeatedly extract connected induced bipartite subgraphs of order at least $k$.  By the observations above the chromatic number of the subgraph induced on all of the extracted vertices is at most $2 C_{\alpha}t$. Thus it remains to bound the chromatic number of the subgraph remaining once the extraction has to stop.

Let $b(G)$ be the largest number of vertices in a connected induced bipartite subgraph of $G$. We have reduced our problem to graphs with $b(G) < k$. In particular, we can assume that $G$ has no induced paths or stars on at least $k$ vertices.

\vskip 5pt

At this point, we employ the Gy\'{a}rf\'{a}s path argument. It was first used by Gy\'{a}rf\'{a}s~\cite{Gyarfas} to show that the chromatic number of graphs with no induced $\ell$-vertex path and no $K_k$ subgraph is bounded by a function of $k$ and $\ell$. Since then variations of this argument have been used extensively to study structure and coloring properties of graphs with forbidden induced subgraphs. To the best of our knowledge, the Gy\'{a}rf\'{a}s path argument has not been previously applied to relaxations of Hadwiger's conjecture.

In its standard form, which we use here, the Gy\'{a}rf\'{a}s path argument grows an induced path with vertex set $v_1,v_2,\ldots,v_i$ in a graph $G$ while retaining a component $C_i$ with high chromatic number in a graph $G_i = G -N_G[v_1] - \ldots - N_G[v_{i-1}]$ containing a neighbour of $v_i$. To continue the growing process, we consider the graph $C_i  -  N_G[v_i]$, choose a component $C_{i+1}$ in it with highest chromatic number and set $v_{i+1}$ to be a neighbour of $v_i$ in $C_i$ with a neighbour in $C_{i+1}$. Bounding the length of the path and the decay $\chi(C_i) - \chi(C_{i+1})$ of the chromatic number allows one to bound $\chi(G)$.

The bound $b(G) < k$ bounds the length of any induced path in $G$ by $k$, but the decay of the chromatic number appears harder to bound directly. Here another novel idea enters the proof. We bound the chromatic number of the graph $J = G[N_G[v_1] \cup \ldots \cup N_G[v_{i}]]$ induced by the deleted vertices (a.k.a., the closed neighborhood of the path). By the bound on $b(G)$ it is not hard to see that this graph has fairly small independence number, i.e.  $\alpha (J) \leq k^2$. On the other hand, we have $\alpha(J) \geq \frac{\v(J)}{\chi_f(J)} \geq \frac{\v(J)}{2\had(J)} \geq \frac{\v(J)}{2t}$ by Theorem~\ref{thm:RS}. Thus $\v(J)$ is small enough to apply the linear bound $\mathrm{LH}(\alpha)$ to $J$, and we can assume that the chromatic number of $J$ is small compared to that of $G$. (This is our second and last application of $\mathrm{LH}(\alpha)$.)

If the process of growing the  Gy\'{a}rf\'{a}s path has stalled, then the chromatic number of the corresponding sequence of components has decayed to zero. Which part of $G$ has the large chromatic number then? It has to come from a component of $G  -  V(J)$ that at some point in the process did not get chosen because there existed another component with the chromatic number at least as high. By itself this observation does not seem particularly helpful, but it fits perfectly into the chromatic-separability framework introduced by Delcourt and Postle, which we now describe.  

\vskip 5pt

For $s \geq 0$ we say that a graph $G$ is \emph{$s$-chromatic-separable} if there exist two vertex disjoint subgraphs $H_1,H_2$ of $G$ such that $\chi(H_i) \geq \chi(G) - s$ for $i=1,2$, and we say that $G$ is  \emph{$s$-chromatic-inseparable}, otherwise. Note that every graph $G$ with $\chi(G) \leq s$ is trivially $s$-chromatic-separable. 

The above Gy\'{a}rf\'{a}s path argument can be formalized to imply that $G$ can not have a $(Dt)$-chromatic inseparable induced subgraph for sufficiently large $D$. (See Corollary~\ref{cor:separation} for a stronger statement that we actually need.) We can additionally assume that the chromatic number of $G$ is much much larger than $Dt$. In particular, $G$ can be split into three vertex disjoint subgraphs $J_1,J_2,J_3$ that retain most of its chromatic number. 

An easily verified equivalent condition for the property that a graph $G$ has a complete graph $K_h$ for some integer $h\ge 1$ as a minor is that it contains a collection of $h$ induced connected subgraphs which are pairwise disjoint and pairwise connected by at least one edge in $G$. This is also called a \emph{$K_h$-model} and the vertex-sets of the disjoint connected subgraphs are referred to as \emph{branch-sets}. This view gives rise to the natural idea of trying to combine several smaller, disjoint such complete minor models, obtained from the disjoint high-chromatic subgraphs $J_1, J_2, J_3$, by linking them together with disjoint paths. To exemplify this a bit more concretely, assume now for simplicity that $t$ is divisible by $3$, i.e. $t = 3t'$ for some $t'$ and that further \begin{itemize}
    \item  $h(J_i) \geq 2t'$ for every $i$, which we might hope to be true by induction on $t$, as $\chi(J_i) \approx \chi(G),$ and
    \item we can ``glue'' together the corresponding three $K_{2t'}$ minor models by constructing a family of $t=3t'$ vertex disjoint paths, each joining a branch set in one model to a branch set in another, with each branch set containing exactly one path endpoint and all path interiors disjoint from the three models.
\end{itemize} 
Then contracting each path together with the two branch sets it joins to a single vertex we would obtain a $K_t$ minor in $G$, contradicting the condition $h(G) < t$.

Satisfying the above conditions seems highly non-trivial, and it is. But fortunately for us this is exactly what Delcourt and Postle accomplish in~\cite{DP}. To facilitate the discussion of the differences between their result and the result we need, let us now introduce the main result of~\cite{DP}, which implies Theorem~\ref{thm:DP}, in its precise form.

\begin{theorem}[Delcourt--Postle, {\cite[Theorem~1.6]{DP}}]\label{thm:dp2}
There is an absolute integer $C_{\rm DP}\ge1$ such that, for every integer
$t\ge3$ and every $K_t$-minor-free graph $G$,
\[
 \chi(G)\le C_{\rm DP}\,t(1+f_{\rm DP}(G,t)),
\]
where
\[
 f_{\rm DP}(G,t)=\max\left(\{0\}\cup\left\{\frac{\chi(H)}a:
 \begin{array}{l} a\in\mathbb Z,\quad t/\sqrt{\log t}\le a\le t,\\
 H \text{ is a subgraph of } G, h(H) < a ,\\
\v(H)\le C_{\rm DP}a(\log a)^4
 \end{array}\right\}\right).
\]
\end{theorem}

The proof of Theorem~\ref{thm:dp2} via a very delicate implementation of the trisection argument described above accomplishes exactly what we need. It shows that if $\chi(G) > C \had(G)$ for some very large $C$, then it contains an $s$-chromatic-inseparable subgraph $H$ with $\chi(H) > C' \had(H)$ with $C'$ smaller than $C$, but still large, and $s$ proportional to $\chi(H)$. 

Note that this is not obvious from the statement. The statement guarantees instead a subgraph $H$ as above with $O(t \log^4 t)$ vertices. The reason behind this size condition is exactly chromatic inseparability. Delcourt and Postle handle the inseparable case by starting with a relatively small highly connected subgraph and growing a minor in it extending the subgraph step by step. The size condition ensures that this process can not exhaust all of $V(G)$ too early.

For us however a highly-chromatic-inseparable subgraph already suffices, and so we recast Theorem~\ref{thm:dp2} as follows. 

\vskip 5pt
For each integer $t\ge3$ and a non-null graph $G$ define
$$
f(G,t)=\sup
\left\{\frac{s}{a} :\begin{array}{l} s \geq 0, \quad
	t/\sqrt{\log t} < a\le t,\\
	H\text{ is an induced subgraph of }G,\quad \had(H) < a,\\
	H\text{ is } s\text{-chromatic-inseparable}
\end{array}\right\}.
$$
The above supremum is well-defined, as $s=0$, $a=t$ and a one-vertex subgraph $H$ satisfy the conditions above. Moreover, as every graph $H$ is trivially $s$-chromatic-separable for every $s \geq \chi(H)$, the ratio $s/a$ in the definition of $f(G,t)$ is upper bounded by $\frac{\sqrt{\log t}\chi(G)}{t}$. When $G$ is a null graph, let $f(G,t)=0$, for every $t$, by convention.

\begin{theorem}\label{thm:DPinsep}
	There exists $C_{\mathrm{sep}} > 0$ such that for every integer $t\ge3$ and
	every graph $G$ with $\had(G) < t$ we have
	\[
	\chi(G)<C_{\mathrm{sep}}\,(1+f(G,t))t.
	\]
\end{theorem}

Theorem~\ref{thm:DPinsep} completes the puzzle reducing the proof of the bootstrap step to chromatic inseparable graphs, which are handled as described above.
The proof of Theorem~\ref{thm:DPinsep} repeats the proof of Theorem~\ref{thm:dp2} essentially verbatim and is presented in Appendix~\ref{app:DPinsep}.

\subsection{Proof of Theorem \ref{thm:bootstrap}}

In this subsection we implement the outline described above. We start with the Gy\'{a}rf\'{a}s path argument.

\begin{lemma}\label[lemma]{lem:path}
Let $s,k \geq 0$ be real. Let $H$ be an $s$-chromatic-inseparable graph with $E(H) \neq \emptyset$ and   $b(H)\le k$. Then there exists an induced subgraph $J$ of $H$ such that
\begin{equation}\label{eq:path}
 \chi(J) > s  \quad \mathrm{and} \quad
 \alpha(J)\le k(k-1). 
\end{equation}
\end{lemma}
\begin{proof}  As $E(H) \neq \emptyset$, we have $k \geq b(H) \geq 2$. A \emph{Gyarfas pair} in $H$ is a pair $(U,v)$ such that \begin{itemize} \item [(G1)]  $U \subseteq V(H)$, $v \in V(H) - U$, \item[(G2)]  $H[U \cup \{v\}]$ is a path with an end $v$, and \item [(G3)] if $v$ has no neighbour in a component $C$ of $H  -  N_H[U]$ then $\chi(C) < \chi(H)-s.$ \end{itemize}
	
As $H$ is $s$-chromatic-inseparable there exists a unique component $C_0$ of $H$ such that $\chi(C_0) \geq \chi(H)-s. $ Then $(\emptyset,v)$ is a Gyarfas pair for every $v \in V(C_0)$. In particular, there exists a Gyarfas pair in $H$.

We choose a Gyarfas pair $(U,v)$ in $H$ with $|U|$ maximum and let $J=H[N_H[U \cup \{v\}]]$. We claim that $J$ satisfies \eqref{eq:path}.
Suppose first, for a contradiction, that $\chi(J) \leq s$. Then $$\chi(H  -  V(J))  \geq \chi(H) - \chi(J) \geq \chi(H)-s.$$ Let $C$ be a component of $H  -  V(J)$ with $\chi(C) \geq \chi(H)-s$. By $s$-chromatic-inseparability of $H$, such a component $C$ is unique. It is contained in a component $C'$ of $H  -  N_H[U]$ with $\chi(C') \geq \chi(C) 
\geq \chi(H)-s.$ By (G3), $v$ has a neighbour in $V(C')$. As $C$ is a component of   $C'  -  N_H(v)$ there exists $v' \in V(C') \cap N_H(v) \subseteq V(H) - N_H[U]$ with a neighbour in $C$.   

We claim that $(U \cup \{v\}, v')$ is a Gyarfas pair. Clearly (G1) holds for $(U \cup \{v\}, v')$. As $v'$ has no neighbours in $U$, and $vv' \in E(H)$, (G2) holds. Finally, as $C$ is the unique component of  $H  -  N_H[U \cup \{v\}]$ with $\chi(C) \geq \chi(H)-s$, and $v'$ has a neighbour in $C$, (G3) holds. 

The Gyarfas pair $(U \cup \{v\}, v')$ contradicts the choice of $(U,v)$. This contradiction implies that $\chi(J) > s$.

It remains to show that $\alpha(J) \leq k(k-1)$. 
Let $P = H[U \cup \{v\}]$. As $P$ is a path, and thus connected and bipartite, we have $\v(P) \leq k.$ Let $S$ be an independent set in $J$  with $|S| = \alpha (J)$. For every $u \in V(P)$ the set $\{u\} \cup (S \cap N_H(u))$ induces a connected and bipartite subgraph in $H$, and so $|S \cap N_H(u)| \leq k-1$. 

If $P$ consists of one vertex $u$ and $u\in S$, then
$|S|=1 \le k(k-1)$. Otherwise, every vertex of $S$ has a neighbour
in $P$ and $$ \alpha(J) = |S| \leq \sum_{u \in V(P)}|S \cap N_H(u)| \leq (k-1)\v(P) \leq k(k-1),$$
as desired. 	
\end{proof}

Combining Lemma~\ref{lem:path} with the Reed-Seymour fractional bound we obtain the following useful consequence.

\begin{corollary}\label{cor:separation}  Let $a,s,k \geq 0$ be real. Suppose that for every graph $J$ with $\had(J) < a$  and $\v(J) \leq 2ak^2$ we have $\chi(J) \leq s$.
Let  $H$ be an $s$-chromatic-inseparable graph with $E(H) \neq \emptyset$ and $b(H) \le k$.
Then $\had(H) \geq a$.
\end{corollary}
\begin{proof} By Lemma \ref{lem:path}  there exists an induced subgraph $J$ of $H$ such that 
$
	\chi(J) > s $ and $
	\alpha(J)\le k(k-1).$ 
Suppose for a contradiction that $\had(H) 
< a$.	
As  $\had(J) \leq \had(H)$, we have $$\v(J) \leq \chi_f(J)\alpha(J) \leq 2\had(J)\alpha(J) \leq 2ak^2,$$
where the second inequality uses Theorem~\ref{thm:RS}. 
Thus, $\chi(J) \leq s$ by the corollary assumptions, a contradiction.
\end{proof}

We are now ready to finish the proof of Theorem~\ref{thm:bootstrap}, which we restate for convenience.

\begin{bootstraptheorem}
For every $\alpha>0$, we have that $\mathrm{LH}(\alpha)$ implies $\mathrm{LH}(\frac{4}{3}\alpha)$. 
\end{bootstraptheorem}
\begin{proof}

By $\mathrm{LH}(\alpha)$  there exist $C_\alpha, t_0 \in \mathbb{R}_{\ge 1}$ such that for every $t \geq t_0$ every graph $G$ with $h(G)<t$ and $\v(G)\le t\log^{\alpha} t$ satisfies $\chi(G)\le C_\alpha t$.

Let $$C' = (2C_\alpha + C_{\mathrm{sep}}(1+C_\alpha)),$$ where $C_{\mathrm{sep}}$ is as in Theorem~\ref{thm:DPinsep}. 
In the following, we will show that 
for every  $t\in \mathbb{N}$ sufficiently large, as a function of $C_\alpha,\alpha$ and $t_0$,  every graph $G$ with $h(G)<t$ and $\v(G)\le t\log^{4\alpha/3} t$ satisfies $\chi(G)\le C' t$, i.e. $\mathrm{LH}(\frac{4}{3}\alpha)$ holds. At the expense of increasing the constant from $C'$ to $2C'$, this then also establishes the statement for all large enough real-valued $t$.

Let $k= (\log t)^{\alpha/3}$. 
We start by extracting large bipartite subgraphs from $G$ as follows.
Let $H_1,\ldots,H_m$ be a collection of pairwise vertex disjoint induced connected bipartite subgraphs of $G$ such that $\v(H_i) > k$ for $1 \le i \le m$, chosen maximal subject to these conditions. Let $U_1 = \bigcup_{i=1}^{m}V(H_i)$ and let $U_2 = V(G) - U_1$. Let $G_1 = G[U_1]$ and $G_2 = G[U_2]$. Note that by the choice of $H_1,\ldots,H_m$ we have $b(G_2) \leq k$. We upper bound $\chi(G_1)$ and $\chi(G_2)$ separately.

An upper bound on $\chi(G_1)$ is derived from the local linear bound at exponent $\alpha$ as  follows. Let $G'_1$ be a minor of $G_1$ obtained by contracting each $H_i$ to a single vertex $w_i$. Then $$\v(G'_1) \leq \frac{\v(G_1)}{k} \leq \frac{\v(G)}{k} \leq t (\log t)^{\alpha},$$   
and $\had(G'_1) \leq \had(G)<t$.  Assuming $t \geq t_0$, we have $\chi(G'_1) \leq C_\alpha t$.

By Lemma~\ref{lem:contraction} we have
\begin{equation}\label{eq:G1} \chi(G_1) \leq 2 \chi(G_1') \leq 2 C_\alpha t.\end{equation} 

The upper bound on $\chi(G_2)$ uses Theorem~\ref{thm:DPinsep} where Corollary~\ref{cor:separation} is used to bound $f(G_2,t)$. Let  $s,a$ and $H$ be chosen to satisfy the conditions in $f(G_2,t)$. That is $s \geq  0,$ $t/\sqrt{\log t} < a \leq t,$ $H$ is an induced subgraph of $G_2$, $\had(H)< a$, $H$  is $s$-chromatic-inseparable. 

As $b(H) \leq b(G_2) \leq k$, Corollary~\ref{cor:separation} would contradict the condition $\had(H) < a$, if it is applicable. Thus it is not, and  there exists a graph $J$ with $\had(J) < a$, $\v(J) \leq 2ak^2$ and $\chi(J) > s$.

We assume that $t$ is large enough so that $2ak^2 = 2a (\log t)^{2\alpha/3} \leq a (\log a)^{\alpha}$ and $a \geq t_0$. Thus as $\had(J) < a$ and $\v(J)\le a(\log a)^\alpha$, we have $\chi(J) \leq  C_\alpha a$. It follows that $s <  C_\alpha a$. As this conclusion holds for all $s,a$ and $H$ satisfying the conditions in the definition of $f(G_2,t)$  it follows that $f(G_2,t) \leq   C_\alpha $.

Theorem~\ref{thm:DPinsep} now implies  \begin{equation}\label{eq:G2} \chi(G_2) \leq C_{\mathrm{sep}}(1+f(G_2,t))t \leq C_{\mathrm{sep}}(1+C_\alpha)t. \end{equation}
Combining \eqref{eq:G1} and \eqref{eq:G2} we have
$$\chi(G) \leq \chi(G_1) + \chi(G_2) \leq (2 C_\alpha  + C_{\mathrm{sep}}(1+ C_\alpha))t  = C't,$$
 as desired.
\end{proof}

\section{Concluding remarks}\label{sec:conc}

\paragraph*{Strengthenings.}
Theorem~\ref{thm:linear} establishes the truth of one of the most prominent relaxations of Conjecture~\ref{con:hadwiger}, contrasting recent results disproving some natural strengthenings~\cite{illingworth2026disproofdominatinghadwigerconjecture,kühn2025disproofoddhadwigerconjecture}.

Linear versions of several such strengthenings remain open, but one  is settled by combining Theorem~\ref{thm:linear}  with a result of the second author. Given graphs $G$ and $H$ we say that $G$ has \emph{an odd $H$ minor} if a graph isomorphic to $H$ can be obtained from a subgraph $G'$ of $G$ by contracting a set of edges forming a cut in $G'$. Let $h_o(G)$ denote the maximum integer $h$ such that the graph $G$ contains $K_h$ 
as an odd minor. 
Gerards and Seymour (see \cite[p. 115]{MR1304254}) conjectured that $\chi(G) \leq h_o(G)$ for every graph $G$. This conjecture has been very recently disproved in~\cite{kühn2025disproofoddhadwigerconjecture}, but the second author~\cite{MR4379303} proved that if $f$ is any function such that $\chi(G) \leq f(h(G))$ for every graph $G$, then every graph $G$ satisfies $\chi(G)\le 2f(h_o(G))$. Thus Theorem~\ref{thm:linear} implies the following.

\begin{corollary}
There is an absolute constant $C$ such that every graph $G$ satisfies $\chi(G)\le C\cdot h_o(G)$. 
\end{corollary}

Another well-studied strengthening of Hadwiger's conjecture generalizes the notion of coloring. A graph $G$ is said to be \emph{$k$-list colorable} if for every assignment of lists $\{L(v)\}_{v \in V(G)}$ to vertices of $G$ such that $|L(v)| \geq k$ for every $v \in V(G)$, there is a choice of colors $\{c(v)\}_{v \in V(G)}$ such that $c(v) \in L(v)$, and $c(v) \neq c(u)$ for every $uv \in E(G)$.  Voigt~\cite{Voigt93} has shown that there exist planar graphs which are not $4$-list colorable, disproving an extension of  Hadwiger's conjecture for list colorings. The linear list Hadwiger conjecture however remains open and has been explicitly stated by Kawarabayashi and Mohar.

\begin{conjecture}
[\cite{KawMoh07}]\label{c:ListHadwiger} There exists a positive constant $C$ such that for every integer $t \geq 1$, every graph with no $K_{t}$ minor is $Ct$-list colorable. 
\end{conjecture}

It is known that if Conjecture~\ref{c:ListHadwiger} holds, then necessarily $C\ge 2$~\cite{MR2861411,MR4496022}, the best known asymptotic bound currently is $O(t\log\log t)$~\cite{GuXu2026}. Several arguments in the current paper extend to list-coloring, but crucially Lemma~\ref{lem:contraction} which allows us to bound the chromatic number of a graph by a linear function of the chromatic number of a much smaller minor appears to have no analogue for list-colorings. This prevents direct extension of the bootstrap argument to this setting.

\medskip 

Finally, another bold strengthening has been recently proposed by Illingworth and Wood~\cite{IllingworthWood2025}. For a graph $G$, a \emph{dominating $K_t$-model} is a sequence $(X_1,\dotsc,X_t)$ of pairwise disjoint non-empty subsets of $V(G)$ such that each $G[X_i]$ is connected and every vertex of $X_j$ has a neighbour in $X_i$ whenever $i<j$. Note that replacing ``every vertex of $X_j$'' by ``some vertex of $X_j$'' recovers the standard definition of a $K_t$-model. 

The Dominating Hadwiger's conjecture asserted that every graph with no dominating $K_{t}$ model is $(t-1)$-colorable. It was recently disproved in~\cite{illingworth2026disproofdominatinghadwigerconjecture}. It would be interesting to see whether any of the ideas in the current paper can be used to improve the upper bounds on the chromatic number of graphs with no dominating $K_t$ model.

\medskip 

We summarize the status of the strengthenings  mentioned above in the following table.  The last row shows the best known upper bound on the (list) chromatic number of a graph without corresponding $K_t$ minor or model. The entries in bold are consequences of the results in this paper.

\begin{table}[htbp]
\centering
\renewcommand{\arraystretch}{1.25}
\begin{tabular}{lcccc}
\hline
 & Original & Odd & List & Dominating \\
\hline
Exact
 & Open
 & False~\cite{kühn2025disproofoddhadwigerconjecture}
 & False~\cite{Voigt93}
 & False~\cite{illingworth2026disproofdominatinghadwigerconjecture}
\\
Linear
 & \textbf{True}
 & \textbf{True}
 & Open
 & Open
\\
Best bound
 & $\boldsymbol{O(t)}$
 & $\boldsymbol{O(t)}$
 & $O(t\log\log t)$~\cite{GuXu2026}
 & $O(t(\log t)^2)$~\cite{GNTT2026}
\\
\hline
\end{tabular}
\label{tab:hadwiger-variants}
\end{table}

We believe that the Linear List Hadwiger conjecture is most likely true, while the Linear Dominating Hadwiger conjecture is most likely false. We are undecided on Hadwiger's conjecture.

\paragraph*{The constant.}

It is natural to wonder what value of $C$ can be guaranteed by the presented proof approach for Theorem~\ref{thm:linear}. We made no effort to optimize this value. A rough estimate, using Kostochka's bound (see Lemma~\ref{lem:density}) for small $t$, indicates one may take $C$ to be about $10^{100}$. There is a lot of room for improvement, and we see no clear barriers to reducing the value of $C$ below $10^{10}$. Proving Theorem~\ref{thm:linear} with, say, $C=100$ would likely require substantial new ideas.

\bibliographystyle{abbrv}
\bibliography{references.bib}

\appendix
\numberwithin{theorem}{section}

\section{A quantitative version of the Pippenger--Frankl--R\"{o}dl theorem}\label{app:matching}

In this section, we give a proof of Theorem~\ref{thm:pippenger}. The Theorem is proved by an application of the R\"{o}dl nibble method introduced in 1985~\cite{RNibble}. A  typical application of this semi-random method constructs a  matching (or a similar discrete structure) in several rounds. In each round a small portion of a matching is randomly constructed, while the rest of the hypergraph is cleaned up to ensure that it is disjoint from the constructed matching and satisfies any regularity, density or other conditions, which we wish to control between rounds. 

We reframe this procedure by packaging a step of it in a single deterministic lemma to avoid the need to simultaneously consider probability distributions introduced at each round. This simplification comes with a trade-off: We will need to maintain a fairly artificial dependence between the parameters of the hypergraph that we start with and that we produce in the lemma. 
 
We start by introducing the necessary notation. For the remainder of the appendix let the integer $r \geq 2$ be fixed and we assume that every hypergraph is $r$-uniform and is edge-weighted with nonnegative edge weights.

For a weighted hypergraph $\mc{H}$ with weights $(x_e)_{e \in E(\mc{H})}$,  let
\begin{align*} \e(\mc{H}) = \sum_{e \in E(\mc{H})}x_e, \qquad 
  d_{\mc{H}}(v)=\sum_{e\ni v}x_e,\qquad
  d_{\mc{H}}(u,v)=\sum_{e\supseteq\{u,v\}}x_e, \\
\Delta(\mc{H}) = \max_{v \in V(\mc{H})}d_{\mc{H}}(v) \quad \text{and} \quad \Delta_2(\mc{H}) = \max_{\substack{u,v \in V(\mc{H})\\u \neq v}} d_{\mc{H}}(u,v)\end{align*}
denote the total edge weight, degrees, codegrees and maximum degree and codegree of $\mc{H}$, respectively. For a hypergraph  with at most one vertex let $\Delta$ and $\Delta_2$ be equal to zero by convention. 

For  $D>0$ and a hypergraph $\mc{H}$ with  $\Delta(\mc{H}) \leq D$, define
\begin{equation}\label{eq:app3-deficit}
  \deff(D,\mc{H})
  =\sum_{v\in V(\mc{H})}\left(1-\frac{d_{\mc{H}}(v)}D\right)
  =\v(\mc{H})-\frac rD \e(\mc{H})\ge0.
\end{equation}
Thus $\deff(D,\mc{H})$ is a measure for the proportion of missing weight of hyperedges in $\mc{H}$ as compared to a weighted hypergraph on the same vertex set with all weighted degrees equal to $D$.

The only somewhat non-standard probabilistic tool  used in our proof is the following bound on the variance of a function of independent random variables in terms of variances of coordinates.

\begin{proposition}[Efron--Stein inequality {\cite[Proposition~1]{BBLM}}]
\label{prop:app3-efron-stein}
Let $Y_1,\ldots,Y_m$ be independent random variables, and let
$Y'_1,\ldots,Y'_m$ be independent copies, independent also of
$Y_1,\ldots,Y_m$. Let $Z=f(Y_1,\ldots,Y_m)$ be square-integrable,
and put
\[
  Z^{(i)}=f(Y_1,\ldots,Y_{i-1},Y'_i,Y_{i+1},\ldots,Y_m).
\]
Then
\[
  \operatorname{Var}Z
  \le\frac12\sum_{i=1}^m\mathbb E\bigl[(Z-Z^{(i)})^2\bigr].
\]
\end{proposition}

We need a specialization of Proposition~\ref{prop:app3-efron-stein} to Bernoulli random variables.

\begin{corollary}
	\label{cor:variance}
	Let $Y_1,\ldots,Y_m$ be independent Bernoulli random variables, and let $p_i = \mathbb E[Y_i]$. Let $Z=f(Y_1,\ldots,Y_m)$. 
	Suppose that $$|f(y_1,\ldots,y_{i-1},y_i,y_{i+1},\ldots,y_m)-f(y_1,\ldots,y_{i-1},1-y_i,y_{i+1},\ldots,y_m)| \leq L_i$$
	for all $y_1,\ldots,y_m \in \{0,1\}$ and every $i \in [m].$ 
	Then\begin{equation}\label{eq:app3-bernoulli-variance}
		\operatorname{Var}Z\le\sum_{i=1}^m p_i(1-p_i)L_i^2.
	\end{equation}
\end{corollary}

\begin{proof}
	Let $Y'_i$ be an independent copy of $Y_i$. Then $\Pr(Y_i\ne Y'_i)=2p_i(1-p_i)$, and so $\mathbb E\bigl[(Z-Z^{(i)})^2\bigr] \leq 2p_i(1-p_i) L_i^2$.
	Thus Proposition~\ref{prop:app3-efron-stein} implies  \eqref{eq:app3-bernoulli-variance}.
\end{proof}	

We are now ready to prove our main technical lemma.

\begin{lemma}\label{lem:app3-one-round}
Let $D > 0$, $0 < a \leq \frac{1}{2r}$ be real, and let $\mc{H}$ be a hypergraph with weights $(x_e)_{e \in E(\mc{H})}$ such that
\begin{equation}\label{eq:app3-round-hypothesis}
  \Delta(\mc{H}) \leq D \qquad \text{and} \qquad \Delta_2(\mc{H})\le\frac{a^3}{64r^2} D.
\end{equation}
Then there exist a matching $M$ in $\mc{H}$
and a subhypergraph $\mc{H}'$ of $\mc{H}$ with weights $(x'_e)_{e \in E(\mc{H}')}$ on a vertex set
$V(\mc{H}')\subseteq V(\mc{H})\setminus V(M)$ with the following properties. \begin{itemize}
	\item $x'_e \leq x_e$ for every $e \in E(\mc{H}')$,
	\item $\Delta(\mc{H}') \leq De^{-a(r-1)}$ 
    and
	\item the set $Z$ of vertices of $\mc{H}$ not covered by either $V(\mc{H}')$ or $V(M)$ satisfies
	\begin{equation}\label{eq:app3-round-conclusion}
		|Z| + 8ar\,\v(\mc{H}')+\deff(De^{-a(r-1)},\mc{H}')	\le 8are^{-a/2}\v(\mc{H}) + \deff(D,\mc{H}). 
	\end{equation} 
\end{itemize}
\end{lemma}

\begin{proof}[Proof of Lemma~\ref{lem:app3-one-round}]
Let $n = \v(\mc{H})$, $D' = De^{-a(r-1)}, \Delta_2 = \Delta_2(\mc{H})$ and write $d(v)$ and $d(u,v)$ instead of $d_{\mc{H}}(v)$ and $d_{\mc{H}}(u,v)$ for brevity.

We select $M$ and $\mc{H}'$ at random as follows.
Let $M_0 \subseteq E(\mc{H})$ be a set of hyperedges selected independently at random with
\[
  \Pr (e \in M_0)  = 1-\exp\left(-\frac{ax_e}{D}\right),
\]
for every $e \in E(\mc{H})$,
and let $Z_0 \subseteq V(\mc{H})$ be a set of vertices 
selected independently at random, and independently from $M_0$, with 
\begin{equation}\label{eq:Z0}
  \Pr (v  \in Z_0)=   1-\exp\left(-a\left(1-\frac{d(v)}D\right)\right) 
\end{equation}
for every $v \in V(\mc{H})$. Thus the hyperedges with high weight are more likely to be selected in $M_0$, while the vertices of low weighted degree are more likely to be selected in $Z_0$.

Let $M$ be the set of hyperedges in $M_0$ which are disjoint from all the remaining  hyperedges in $M_0$.
Let $U = V(\mc{H}) - (V(M_0) \cup Z_0)$. Let $\mc{H}_0$ be the subhypergraph  of $\mc{H}$ induced by $U$
retaining the edge weight in $\mc{H}$. We obtain $\mc{H}'$ from $\mc{H}_0$ by reducing edge weights to ensure $\Delta(\mc{H}') \leq D'$, as follows. While there exists a vertex $v$ such that $d_{\mc{H}'}(v) > D'$, we reduce the weights of hyperedges incident to $v$ until we reduce $d_{\mc{H}'}(v)$ to $D'$. The order in which we consider the vertices and hyperedges incident to a vertex is irrelevant.

This finishes the description of the random objects $M$ and $\mc{H}'$. Clearly, they satisfy all the properties required by the lemma except possibly \eqref{eq:app3-round-conclusion}. We will show that  
\begin{equation}\label{eq:app3-E}
		\bb{E} \s{ |Z| + 8ar\,\v(\mc{H}')+\deff(D',\mc{H}')}	\le 8are^{-a/2}n + \deff(D,\mc{H}), 
	\end{equation} 
which immediately implies existence of choices of $M$ and $\mc{H}'$ satisfying the lemma.

First, however, let us  estimate the probabilities that a set of vertices survives in $U$.
For every set $S\subseteq V(\mc{H})$, independence gives
\begin{equation}\label{eq:app3-joint-survival}
\begin{aligned}
 \Pr(S\subseteq U) &=
  \prod_{v\in S}\exp\left(-a\left(1-\frac{d(v)}D\right)\right)
  \prod_{e:\,e\cap S\ne\varnothing}
      \exp\left(-\frac{ax_e}D\right)\\
 &=
  \exp\left(-a|S|+\frac aD
       \left(\sum_{v\in S}d(v)
             -\sum_{e:\,e\cap S\ne\varnothing}x_e\right)\right)\\
 &=
  \exp\left(-a|S| +\frac aD
       \sum_e x_e(|e\cap S|-1)_+\right).
\end{aligned}
\end{equation}

By \eqref{eq:app3-joint-survival} we have $\Pr (S\subseteq U) \geq e^{-a|S|}$ for every $S\subseteq V(\mathcal H)$ and $\Pr(s\in U)=e^{-a}$ for every $s\in V(\mathcal H)$. This implies 
\begin{equation}\label{eq:app3-survival-moments}
	\mathbb E\v(\mc{H}') = \mathbb E|U|=e^{-a}n,~~
	\Pr(e\subseteq U)\ge e^{-ar} \text{ for every }e\in E(\mathcal H) \text{ and so } \bb{E}\e(\mc{H}_0)\geq e^{-ar}\e(\mc{H}).
\end{equation}

As $(j-1)_+\le\binom j2$ for integers $j\ge0$, we also have
\begin{equation*}
\begin{aligned}
  \sum_e x_e(|e\cap S|-1)_+
  \le\sum_e x_e\binom{|e\cap S|}{2}
  =\sum_{\{u,v\}\subseteq S}d(u,v)
   \le \frac{|S|^2 \Delta_2}{2},
\end{aligned}
\end{equation*}
and thus 
\begin{equation}\label{eq:app3-survival-correction}
\Pr(S\subseteq U) \leq   \exp\left( -a |S| + \frac {a|S|^2 \Delta_2}{2D}
\right)
\end{equation}
for every $S\subseteq V(\mathcal H)$. Let $W = \e(\mc{H}_0) - \e(\mc{H}')$ be the total weight removed in the process of reducing the edge weights of $\mc{H}_0$ to reach $\mathcal{H}'$.

We will establish \eqref{eq:app3-round-conclusion} by proving the following claims.

\begin{claim}\label{claim:M}
\begin{equation}\label{eq:app3-M}
\bb{E} |M| \geq \frac{ae^{-ar}}{D}\e(\mc{H}).
\end{equation}
\end{claim}

\begin{claim}\label{claim:W}
	\begin{equation}\label{eq:app3-normalized-trimming}
			\mathbb  E W
			\le\frac{D'a^2}{2} n.
\end{equation}
\end{claim}	

Before proving these claims let us derive \eqref{eq:app3-E} (and thus the lemma) from them.
Expanding the left side of  \eqref{eq:app3-E} we obtain
\begin{align*} 
\bb{E} &\s{ |Z| + 8ar\,\v(\mc{H}')+\deff(D',\mc{H}')}\stackrel{\eqref{eq:app3-deficit},\eqref{eq:app3-survival-moments}}{=} \bb{E}|Z| + 8are^{-a}n +  \bb{E}\v(\mc{H}') - \frac{r}{D'}\bb{E}\s{\e(\mc{H}')} \\ &=  (n - r\bb{E}|M|- \bb{E}\v(\mc{H}')) + 8are^{-a}n +  \bb{E}\v(\mc{H}') - \frac{r}{D'}\bb{E}\s{\e(\mc{H}')} \\&=  n + 8are^{-a}n  - r\bb{E}|M| - \frac{r}{D'}\bb{E}\s{\e(\mc{H}_0) - W} \\ & \stackrel{\eqref{eq:app3-survival-moments},\eqref{eq:app3-M},\eqref{eq:app3-normalized-trimming}}{\leq} n + 8are^{-a}n - r \frac{ae^{-ar}}{D}\e(\mc{H}) - \frac{r}{D'}  e^{-ar}\e(\mc{H}) + \frac{ra^2}{2} n.
\end{align*}
The right side of  \eqref{eq:app3-round-conclusion} is equal to
$$ 8are^{-a/2}n + n - \frac{r}{D}\e(\mc{H}).$$
Thus, substituting $D' = De^{-a(r-1)}$ it suffices to show that
\begin{equation}\label{eq:final}
   \frac{r}{D} \s{1 - ae^{-ar}-e^{-a}} \e(\mc{H}) \leq \s{8ar(e^{-a/2}-e^{-a}) -\frac{ra^2}{2}}n
\end{equation}

As \begin{equation}\label{eq:aux} 8(e^{-a/2}-e^{-a}) = 8e^{-a}(e^{a/2}-1)\geq 8 \frac{a}{2}e^{-a} \geq \frac{3}{2}a, \end{equation}
(the last estimate is lossy), the right side of \eqref{eq:final} is lower-bounded by $ra^2n$. As $\e(\mc{H}) \leq \frac{D}{r} n$, \eqref{eq:final} is implied by
\begin{equation}\label{eq:final2}
  1 - ae^{-ar}-e^{-a}\leq 1 - (1-a) - a(1-ar) = ra^2.
\end{equation}
This establishes \eqref{eq:final2}. It remains to prove Claims~\ref{claim:M} and \ref{claim:W}.

\begin{proof}[Proof of Claim~\ref{claim:M}]
For each edge $e$ we have
\[
\begin{aligned}
  \Pr(e\in M) &=
  \left(1-e^{-ax_e/D}\right)
  \prod_{\substack{f\ne e\\f\cap e\ne\varnothing}}
  e^{-ax_f/D}
 \\ & =
  \left(1-e^{-ax_e/D}\right)
  \exp\!\left(-\frac aD
  \sum_{\substack{f\ne e\\f\cap e\ne\varnothing}}x_f\right)
 \\ &=\left(e^{ax_e/D}-1\right)
       \exp\left(-\frac aD\sum_{f:\,f\cap e\ne\varnothing}x_f\right)
  \\ &\ge\frac{ax_e}{D}e^{-ar},
\end{aligned}
\]
since $e^y-1\ge y$ for $y\ge0$ and
$\sum_{f:\,f\cap e\ne\varnothing}x_f
 \le\sum_{u\in e}d(u)\le rD$.
Summing over all $e \in E(\mc{H})$ gives~\eqref{eq:app3-M}.
\end{proof}

\begin{proof}[Proof of Claim~\ref{claim:W}]

This proof is the most technical part of the argument.
Clearly	
\begin{equation}\label{eq:app3-trimming-potential}
  W\le\sum_{v\in U}(d_{\mc{H}_0}(v)-D')_+.
\end{equation}

We estimate the right-hand side by fixing a vertex $v$ and
conditioning on the event $v \in U$. Until we return to $W$ below, all
probabilities, expectations, and variances are under
this conditioning. Let $I_e$ be $1$ if $e\subseteq U$ and $0$
otherwise, and put
\[
  X_v=d_{\mc{H}_0}(v)=\sum_{e\ni v}x_eI_e.
\]
By \eqref{eq:app3-survival-correction} we have
\[
  \mathbb E I_e
  \le  \exp\left( -a (r-1)+ \frac {ar^2 \Delta_2}{2D}\right)
\]
As $\frac {ar^2 \Delta_2}{2D}\le1/2$ by \eqref{eq:app3-round-hypothesis} 
and
$e^y \le 1+ye^y \le1+ 2 y$ for $0\le y\le1/2$, we have
\begin{equation}\label{eq:app3-degree-mean}
\begin{aligned}
  \mathbb E X_v
  &\le d(v)e^{-a(r-1)}\left(1+\frac{ar^2\Delta_2}{D}\right)\\
  &\le d(v)e^{-a(r-1)}+ar^2\Delta_2
   \le D'+ar^2\Delta_2,
\end{aligned}
\end{equation}
where we also used $d(v)\le D$, $e^{-a(r-1)}\le1$, and $D'=De^{-a(r-1)}$.

To upper bound the variance of $X_v$ we apply Corollary~\ref{cor:variance}.
Conditioning on $v\in U$  ensures that $v \not \in Z_0$ and $e \not \in M_0$ for every $e \ni v$. 
The random variable $X_v$ is thus a function of independent random variables $\mathds{1}_{u \in Z_0}$ for $u \neq v$ and  $\mathds{1}_{e \in M_0}$ for $e \not \ni v$. Changing $\mathds{1}_{u \in Z_0}$
changes $X_v$ by at most $d(v,u)$. Changing $\mathds{1}_{e \in M_0}$
changes $X_v$ by at most
\[
  \sum_{\substack{f\ni v\\f\cap e\ne\varnothing}}x_f
  \le \sum_{u\in e}d(v,u).
\]
Since $1-e^{-y}\le y$ for $y\ge0$, we get from \eqref{eq:Z0} that $$ \mathbb E\mathds{1}_{u \in Z_0} \leq a\left(1-\frac{d(u)}D\right)$$
and we similarly have $$\mathbb E\mathds{1}_{e \in M_0} \leq\frac{ax_e}{D}. $$
Plugging these and the estimates above into \eqref{eq:app3-bernoulli-variance} we obtain
\begin{equation}\label{eq:app3-degree-variance}
\begin{aligned}
  \operatorname{Var}X_v
  &\le a\sum_{u\ne v}\left(1-\frac{d(u)}D\right)(d(v,u))^2
       +\frac aD\sum_{e\not\ni v}x_e
          \left(\sum_{u\in e}d(v,u)\right)^2\\
  &\le a\sum_{u\ne v}\left(1-\frac{d(u)}D\right)(d(v,u))^2
       +\frac{ar}D\sum_{u\ne v}d(u)(d(v,u))^2\\
  &\le ar\sum_{u\ne v}(d(v,u))^2
   \le ar\Delta_2\sum_{u\ne v}d(v,u)\\
  &=ar(r-1)\Delta_2d(v)\le ar^2\Delta_2D.
\end{aligned}
\end{equation}
The second line uses Cauchy--Schwarz and
$\sum_{e\ni u,\,e\not\ni v}x_e\le d(u)$.
The last two lines use $d(u)\le D$, $d(v,u) \le \Delta_2$, and
$\sum_{u\ne v}d(v,u)=(r-1)d(v)$.
Using Cauchy-Schwarz again, together with
\eqref{eq:app3-degree-mean} and~\eqref{eq:app3-degree-variance},
we obtain
\begin{align*}
  \mathbb E(X_v-D')_+ &\le \mathbb E|X_v-\mathbb E X_v| + (\mathbb E X_v-D')_+ \\ &\leq \sqrt{\operatorname{Var}X_v}+(\mathbb E X_v-D')_+  \\ &\leq r\sqrt{a\Delta_2D} + ar^2\Delta_2.
\end{align*}
We now remove the conditioning on $v\in U$. Each vertex belongs
to $U$ with probability $e^{-a}$, so~\eqref{eq:app3-trimming-potential}
implies
\begin{align*}\label{eq:app3-expected-trimming}
  \mathbb EW &\le e^{-a}n\bigl(r\sqrt{a\Delta_2D}+ar^2\Delta_2\bigr) \le e^{-a}a^2 Dn \s{\frac{1}{8} + \frac{a^2}{64}} \\ &\leq \frac{1}{4} e^{-a} a^2 Dn \leq \frac{1}{2} a^2 D'n,
\end{align*}
where the second inequality is obtained by substituting an upper bound on $\Delta_2$ from \eqref{eq:app3-round-hypothesis} and the last inequality holds as $e^{-a}D = e^{a(r-2)}D' \leq e^{1/2}D' \leq 2D'.$ Thus \eqref{eq:app3-normalized-trimming} holds, as desired.
\end{proof}
\end{proof}

Iteratively applying Lemma~\ref{lem:app3-one-round} we obtain a matching covering almost all the vertices of $\mc{H}$ under appropriate conditions.

\begin{lemma}
\label{prop:app3-matching}
Let $D>0$ be real, $r\ge 2$ an integer, and let $\mc{H}$ be an edge-weighted $r$-uniform hypergraph with
$\Delta(\mc{H})\le D$.
Then $\mc{H}$ contains a matching leaving at most
\[
  \deff(D,\mc{H})
  +\left(\frac{2^{15}r^5 \Delta_2(\mc{H})}{D}\right)^{\frac{1}{2r}}\v(\mc{H})
\]
vertices uncovered.
\end{lemma}

\begin{proof}
If $\Delta_2(\mc{H})=0$, then all edge weights are zero, so
$\deff(D,\mc{H})=\v(\mc{H})$ and the empty matching suffices. 
Otherwise, by rescaling we may assume $\Delta_2(\mc{H})  = 1$. We prove that every hypergraph $\mc{H}$  with
$\Delta(\mc{H})\le D$ and $\Delta_2(\mc{H}) \leq 1$ contains a matching   leaving at most
\[
  \deff(D,\mc{H})
  +\left(\frac{2^{15}r^5}{D}\right)^{\frac{1}{2r}}\v(\mc{H})
\]
vertices uncovered by induction on $\lceil D\rceil$.

Let
\[
  \varepsilon=\left(\frac{2^{15}r^5}{D}\right)^{\frac{1}{2r}}.
\]
If $\varepsilon\ge1$, the empty matching suffices.
Otherwise, let  $a=\frac{\varepsilon}{8r}$. 
We have
\[
 D
  =\frac{2^{15}r^5}{\varepsilon^{2r}}
  \ge\frac{2^{15}r^5}{\varepsilon^3}
  =\frac{64r^2}{a^3}.
\]
Moreover,
\[
  D- e^{-a(r-1)}D \ge(1-e^{-a})D \ge \frac{aD}{2}\ge 1.
\]

Apply Lemma~\ref{lem:app3-one-round} to $\mc{H}$ to obtain a matching $M$ in $\mc{H}$ and a subhypergraph $\mc{H}'$ of $\mc{H}$
with $V(\mc{H}')\subseteq V(\mc{H})\setminus V(M)$, $\Delta(\mc{H}') \leq De^{-a(r-1)}$, $\Delta_2(\mc{H}') \leq \Delta_2(\mc{H}) \leq 1$, 
such that at most
$$
 \eps e^{-a/2}\v(\mc{H}) + \deff(D,\mc{H}) - \eps\,\v(\mc{H}')- \deff(De^{-a(r-1)},\mc{H}')
$$
vertices of $\mc{H}$ are covered by neither $M$ nor $\mc{H}'$.
The induction hypothesis, applied to $\mc{H}'$ with degree bound
$e^{-a(r-1)}D$, gives a matching in $\mathcal {H}'$ leaving
at most
$$ 
  \left(\frac{2^{15}r^5}{De^{-a(r-1)}}\right)^{\frac{1}{2r}}\v(\mc{H}')+\deff(De^{-a(r-1)},\mc{H}') = \eps e^{\frac{a(r-1)}{2r}}\v(\mc{H}')+\deff(De^{-a(r-1)},\mc{H}') $$ vertices of $\mc{H}'$ uncovered.
Its union with $M$ is a matching in $\mc{H}$ leaving at most
$$  \eps e^{\frac{a(r-1)}{2r}}\v(\mc{H}') -\eps \v(\mc{H}') + \eps e^{-a/2} \v(\mc{H}) + \deff(D,\mc{H}) \leq \eps\s{e^{\frac{a(r-1)}{2r}} + e^{-a/2} -1}  \v(\mc{H}) + \deff(D,\mc{H}) 
$$
vertices uncovered.
It therefore suffices to show that 
\[
  e^{-a/2}+e^{\frac{a(r-1)}{2r}}\le2.
\]
As $e^{-t}\le 1-t+t^2/2$ for $t\ge0$ and
$e^t \le 1+t+t^2$ for $0\le t\le1/2$, we obtain
\[
\begin{aligned}
  e^{-a/2}+e^{\frac{a(r-1)}{2r}}
  &\le
  \left(1-\frac a2+\frac{a^2}{8}\right)
  +\left(1+\frac{a(r-1)}{2r}+\frac{a^2}{4}\right)=2-\frac{a}{2r}+\frac{3a^2}{8}
  \le2,
\end{aligned}
\]
where the last inequality follows from $a<\frac{1}{8r}$.
\end{proof}

\begin{proof}[Proof of Theorem~\ref{thm:pippenger}]
We prove the conclusion for every $0<\mu\le \frac{(\xi/2)^{2r}}{ 2^{15}r^5}$.
This includes $\mu=(\xi/100)^{20r}$,
since $0<\xi\le1$ and $100^{20r}\ge2^{2r+15}r^5$ for $r\ge2$.

By hypothesis, $\Delta(\mc{H})\le1$, $\Delta_2(\mc{H})\le\mu$,
and
\[
  \deff(1,\mc{H})
  =\sum_{v\in V(\mc{H})}(1-d_{\mc{H}}(v))
  \le\mu\v(\mc{H}).
\]
Apply Lemma~\ref{prop:app3-matching} with $D=1$.
Since $2^{15}r^5 \mu\le(\xi/2)^{2r}$, it gives a matching leaving at most
\[
  \deff(1,\mc{H})+(2^{15}r^5\mu)^{\frac{1}{2r}}\v(\mc{H})
  \le\left(\mu+\frac{\xi}{2}\right)\v(\mc{H})
  \le\xi\v(\mc{H})
\]
vertices uncovered. 
\end{proof}

\begingroup
\hypersetup{hidelinks}
\section{Proof of Theorem~\ref{thm:DPinsep}}\label{app:DPinsep}
\theoremstyle{definition}
\newtheorem{dpdefinition}[theorem]{Definition}
\theoremstyle{plain}
\newtheorem{dpclaim}{Claim}[theorem]

\definecolor{DPchange}{RGB}{0,90,40}
\definecolor{DPdelete}{RGB}{165,35,30}
\newcommand{\DPnew}[1]{\begingroup\color{DPchange}#1\endgroup}
\newenvironment{DPaddition}{\par\color{DPchange}}{\par}
\newsavebox{\DPstrikebox}
\newcommand{\DPstrike}[1]{\begingroup
 \sbox{\DPstrikebox}{#1}\leavevmode
 \rlap{\raisebox{.5ex}{\rule{\wd\DPstrikebox}{.45pt}}}%
 \usebox{\DPstrikebox}\endgroup}
\ExplSyntaxOn
\seq_new:N \l_dp_deleted_words_seq
\NewDocumentCommand{\DPdel}{m}
 {\group_begin:\color{DPdelete}
  \seq_set_split:Nnn \l_dp_deleted_words_seq {~} {#1}
  \seq_map_indexed_inline:Nn \l_dp_deleted_words_seq
   {\int_compare:nNnT {##1}>{1}{\space}\DPstrike{##2}}
  \group_end:}
\ExplSyntaxOff
\newcommand{\DPdelmath}[1]{\text{\color{DPdelete}\DPstrike{\ensuremath{#1}}}}

The proof adapts the proofs of~\cite[Theorems 1.6 and 7.1]{DP} with minimal changes.
We first record the necessary definitions and statements used in the proofs.

\subsection{Definitions and results from~\cite{DP}}

We start by introducing the needed additional notation.

We denote by $d(G)=\e(G)/\v(G)$ the density of a non-empty graph $G$. We use
$\kappa(G)$ to denote the (vertex) connectivity of $G$.
 
We say that vertex-disjoint subgraphs $H$ and $H'$ of a graph $G$ are
\emph{adjacent} if there exists an edge of $G$ with one end in $V(H)$ and the
other end in $V(H')$. For a positive integer $n$, let $[n]$ denote
$\{1,2,\ldots,n\}$. A collection $\mathcal X=\{X_1,X_2,\ldots,X_h\}$
of pairwise vertex-disjoint subgraphs of $G$ is a \emph{model} of a graph
$H$ in a graph $G$ if $X_i$ is connected for every $i\in[h]$, and there
exists a bijection $\phi:V(H)\to[h]$ such that $X_{\phi(u)}$ and
$X_{\phi(v)}$ are adjacent for every $uv\in E(H)$. It is well-known
and not hard to see that $G$ has an $H$ minor if and only if there exists
a model of $H$ in $G$. 

We write $V(\mathcal X)$ for
$\bigcup_{i\in[h]}V(X_i)$. We say that a \emph{model} $\mathcal X$ as above
is rooted at $S$ for $S\subseteq V(G)$ if $|S|=h$ and
$|V(X_i)\cap S|=1$ for every $i\in[h]$.

\vskip 10pt
We now list the results which are  used in our adaptation  of the proofs of~\cite[Theorems 1.6 and 7.1]{DP}. (Some of the results used in the original proofs are no longer needed.)

\begin{theorem}[{\cite[Theorem 2.3]{DP}}]
\label{lem:small-connected}
There exists an integer $C=C_{\ref{lem:small-connected}}\ge1$ such that
the following holds: Let $t\ge1$ be an integer. For every integer
$k\ge t$, if $G$ is a graph with $d(G)\ge Ck$ and $G$ contains no
$K_t$ minor, then $G$ contains a non-empty $k$-connected subgraph $H$
with $\v(H)\le C^2\cdot t\cdot\log^3t$.
\end{theorem}


\begin{theorem}[Kostochka~\cite{Kos82}; {\cite[Theorem 4.3]{DP}}]
\label{lem:density}
Let $t\ge2$ be an integer. Then every graph $G$ with
$d(G)\ge30\cdot t\cdot\sqrt{\log t}$ has a $K_t$ minor.
\end{theorem}

\begin{theorem}[Gir\~ao--Narayanan~\cite{GN}; {\cite[Theorem 5.1]{DP}}]
\label{lem:connectivity}
For every positive integer $k$, if $G$ is a graph with $\chi(G)\ge7k$,
then $G$ contains a $k$-connected subgraph $H$ with
$\chi(H)\ge\chi(G)-6k$.
\end{theorem}

Let $G$ be a graph and $A,B\subseteq V(G)$. A path $P=v_1\ldots v_k$
in $G$ is an \emph{$A-B$ path} if $V(P)\cap A=\{v_1\}$ and
$V(P)\cap B=\{v_k\}$. A subset $X\subseteq V(G)$ separates $A$ and
$B$ if every $A-B$ path contains a vertex of $X$.

\begin{theorem}[Menger~\cite{Menger}; {\cite[Theorem 5.2]{DP}}]
\label{lem:menger}
Let $G$ be a graph and $A,B\subseteq V(G)$. Then the minimum number
of vertices separating $A$ from $B$ is equal to the maximum number
of vertex-disjoint $A-B$ paths.
\end{theorem}

\begin{lemma}[{\cite[Lemma 5.3]{DP}}]
\label{lem:redundancy}
Let $G$ be a graph and let $A_1,A_2,B$ be disjoint non-empty subsets
of $V(G)$. If for each $i\in\{1,2\}$, there exists a set
$\mathcal P_i$ of $A_i-B$ paths in $G\setminus A_{3-i}$ with
$|\mathcal P_i|=2|A_i|$ which are pairwise disjoint in
$G\setminus A_i$ such that every vertex of $A_i$ is in exactly two
paths in $\mathcal P_i$, then there exist $|A_1|+|A_2|$
vertex-disjoint $(A_1\cup A_2)-B$ paths in $G$.
\end{lemma}

\begin{dpdefinition}[{\cite[Definition 5.4]{DP}}]
\label{def:linkage}
Let $\ell$ be a positive integer. Let $G$ be a graph and let
$\mathcal S=\{(s_i,t_i)\}_{i\in[\ell]}$ be a collection of pairs of
vertices of a graph $G$ such that for distinct $i,j\in[\ell]$, we
have that $s_i\ne s_j$, $t_i\ne t_j$, and further if $s_i=t_j$ for
some $i,j\in[\ell]$, then $i=j$. An $\mathcal S$-linkage
$\mathcal P$ is a collection of vertex-disjoint paths
$\{P_1,\ldots,P_\ell\}$ in $G$ such that $P_i$ has ends $s_i$ and
$t_i$ for every $i\in[\ell]$. We let $V(\mathcal P)$ denote the set
$\bigcup_{i=1}^{\ell}V(P_i)$.
\end{dpdefinition}

\begin{dpdefinition}[{\cite[Definition 5.5]{DP}}]
\label{def:linked}
A graph $G$ is said to be $\ell$-linked if $1\le\ell\le \v(G)$ and
for any collection $\mathcal S=\{(s_i,t_i)\}_{i\in[\ell]}$ of pairs
of vertices of a graph $G$ such that for distinct $i,j\in[\ell]$,
we have that $s_i\ne s_j$ and $t_i\ne t_j$, and further if $s_i=t_j$
for some $i,j\in[\ell]$, then $i=j$, there exists an
$\mathcal S$-linkage.
\end{dpdefinition}

\begin{theorem}[Bollob\'as--Thomason~\cite{BT}; {\cite[Theorem 5.6]{DP}}]
\label{lem:linkage}
There exists an integer $C=C_{\ref{lem:linkage}}\ge1$ such that the
following holds: Let $\ell$ be a positive integer. If $G$ is a graph
with $\kappa(G)\ge C\ell$, then $G$ is $\ell$-linked.
\end{theorem}

\begin{dpdefinition}[{\cite[Definition 5.12]{DP}}]
\label{def:woven}
Let $a,b$ be nonnegative integers. We say a graph $G$ is
\emph{$(a,b)$-woven} if the following holds: for every three sets of vertices
$R=\{r_1,\ldots,r_a\}$, $S=\{s_1,\ldots,s_b\}$,
$T=\{t_1,\ldots,t_b\}$ in $V(G)$ such that if $s_i=t_j$ for some
$i,j\in[b]$, then $i=j$, there exists a $K_a$ model $\mathcal M$
in $G$ rooted at $R$ and an $\{(s_i,t_i)\}_{i\in[b]}$-linkage
$\mathcal P$ in $G$ such that
$V(\mathcal M)\cap V(\mathcal P)=R\cap(S\cup T)$.
\end{dpdefinition}

\begin{lemma}[Kawarabayashi~\cite{Kaw}; {\cite[Lemma 5.14]{DP}}]
\label{lem:rooted-minor}
Let $a\ge1$ be an integer. If $G$ is a graph with $\kappa(G)\ge a$
and $G$ has a $K_{2a}$ minor, then $G$ is $(a,0)$-woven.
\end{lemma}

\begin{lemma}[{\cite[Lemma 5.15]{DP}}]
\label{lem:weaving}
Let $a,b$ be positive integers. Let $G$ be a graph. Let
$S=\{s_1,\ldots,s_b\}$, $T=\{t_1,\ldots,t_b\}$ be disjoint sets
of vertices in $G$. Let $\mathcal P$ be an
$\{(s_i,t_i)\}_{i\in[b]}$-linkage in $G$. If $H$ is a subgraph of
$G$ with $\v(H)\ge a+2b$ that is $(a,b)$-woven and $R=\{r_1,\ldots,r_a\}\subseteq V(H)$,
then there exists an $\{(s_i,t_i)\}_{i\in[b]}$-linkage
$\mathcal P'$ in $G$ and a $K_a$ model $\mathcal M$ in $G$ rooted
at $R$ such that $V(\mathcal P')\cap V(\mathcal M)\subseteq
R\cap V(\mathcal P)$ and $V(\mathcal P')\subseteq
V(H)\cup V(\mathcal P)$.
\end{lemma}


\subsection{An adaptation of {\cite[Theorem 7.1]{DP}}}

The proof of Lemma~\ref{lem:DP-woven} follows the proof of
\cite[Theorem 7.1]{DP}; the proof of
Theorem~\ref{thm:DPinsep} follows their derivation of Theorem~1.6. To emphasize the similarity of the proofs, we indicate the changes from the original proofs in 
\DPnew{dark green}. The remaining parts of the proofs of Lemma~\ref{lem:DP-woven} and Theorem~\ref{thm:DPinsep},  except for the differences in indexing and one misprint correction, correspond verbatim to passages in the original proofs in \cite{DP}.

\begin{lemma}\label{lem:DP-woven}
There exists an integer $C=C_{\ref{lem:DP-woven}}\ge1$ such that the following
holds: Let $t\ge3$ be an integer. Let $G$ be a graph and let
$f_{7.1}(G,t)=f(G,14t)$, where $f$ is defined in Section~\ref{sec:bootstrap}.
Suppose that $t$ is a power of three and
$a=\left(\frac23\right)^i\cdot t$ is a positive integer for some
nonnegative integer $i$. If $G$ is a $Ca$-connected graph with
\[
 \chi(G)\ge C(t+276\cdot a\cdot(1+f_{7.1}(G,t))),
\]
then $G$ is $(a,3a)$-woven.
\end{lemma}
\begin{proof}\begin{DPaddition}

Let $C_{\ref{lem:DP-woven}}:=35\cdot C_{\ref{lem:small-connected}}^{2}\cdot C_{\ref{lem:linkage}}$.
Choose $C_{\ref{lem:small-connected}}\ge480$. Then
$C_{\ref{lem:DP-woven}}\ge2000$ and $C_{\ref{lem:DP-woven}}\ge98C_{\ref{lem:linkage}}$.
All subgraphs supplied by connectivity extraction or chromatic
separation below are assumed to be  induced, without loss of generality.
\end{DPaddition}

We proceed by induction on $a$. Suppose that $G$ contains a $K_{14a}$ minor. Since
$C_{\ref{lem:DP-woven}}\ge7$, we have that $G$ is $7a$-connected. Then by
Lemma~\ref{lem:rooted-minor}, we have that $G$ is $(7a,0)$-woven. Hence $G$ is
$(a,3a)$-woven as desired. 
So we may assume that $G$ is $K_{14a}$-minor-free. 

\begin{DPaddition} Suppose that $a  \leq \frac{t}{\sqrt{\log t}}.$
Theorem~\ref{lem:density} implies that a $K_r$-minor-free graph has chromatic
number at most $60r\sqrt{\log r}+1$. Consequently, as $G$ has no $K_{14a}$ minor,
then
\[
 \chi(G)\le840a\sqrt{\log(14a)}+1
 \le840t\sqrt{1+\frac{\log14}{\log3}}+1
 <2000t\le C_{\ref{lem:DP-woven}}t,
\]
a contradiction. Thus {$a >\frac{t}{\sqrt{\log t}}$}.
\end{DPaddition}

Now we proceed to show that $G$ is $(a,3a)$-woven directly.
Let $R=\{r_1,\ldots,r_a\}$, $S=\{s_1,\ldots,s_{3a}\}$,
$T=\{t_1,\ldots,t_{3a}\}$ be subsets of $V(G)$ such that if
$s_i=t_j$ for some $i,j\in[3a]$, then $i=j$. Since
$\delta(G)\ge\kappa(G)\ge8a=2|R|+|S|+|T|$ as $C_{\ref{lem:DP-woven}}\ge8$,
we may assume without loss of generality that $R$, $S$ and $T$
are pairwise disjoint.

For all $i\in[3a]$, let $z_i:=s_i$. For all $i\in[a]$, let
$z_{3a+i}:=r_i$. For all $i\in[3a]$, let $z_{4a+i}:=t_i$.
Let $Z:=\{z_i:i\in[7a]\}$. Let $G_1:=G\setminus Z$. Note that
\[
 \chi(G_1)\ge\chi(G)-|Z|\ge\chi(G)-7a
 \ge C_{\ref{lem:DP-woven}}\cdot(t+269\cdot a\cdot(1+f_{7.1}(G,t))).
\]
Let $k:=14\cdot C_{\ref{lem:linkage}}\cdot a$. Since $C_{\ref{lem:linkage}}\ge1$, we have
that $k\ge14a$.
Note that $d(G_1)\ge\frac{\delta(G_1)}2\ge\frac{\delta(G)-|Z|}2$.
Hence we have that
\[
 d(G_1)\ge\frac{\kappa(G)-7a}{2}\ge C_{\ref{lem:small-connected}}\cdot k,
\]
where we used that $C_{\ref{lem:DP-woven}}\ge28\cdot C_{\ref{lem:linkage}}\cdot C_{\ref{lem:small-connected}}+7$.
By Theorem~\ref{lem:small-connected} since $G_1$ is $K_{14a}$-minor-free, there exists a
$(14\cdot C_{\ref{lem:linkage}}\cdot a)$-connected induced subgraph $H_0$ of $G_1$. 

\begin{DPaddition}
If $\chi(H_0)>7k$, delete vertices from $H_0$ until the resulting graph has chromatic number exactly $7k$, and apply
Theorem~\ref{lem:connectivity} to replace $H_0$ by an induced $k$-connected subgraph
of it. Thus we assume
\[
 \kappa(H_0)\ge14C_{\ref{lem:linkage}}a,\qquad
 \chi(H_0)\le7k=98C_{\ref{lem:linkage}}a.
\]
\end{DPaddition}

Since $G$ is $14a$-connected as $C_{\ref{lem:DP-woven}}\ge14$, it follows from
Theorem~\ref{lem:menger} that there exists a set $\mathcal P_1$ of
$Z-V(H_0)$ paths with $|\mathcal P_1|=2|Z|$ which are vertex-disjoint
except in $Z$ and where each vertex in $Z$ is the end of exactly two
paths in $\mathcal P_1$.
We may assume without loss of generality that each path in
$\mathcal P_1$ is induced. Hence
\[
 \chi\left(\bigcup_{P\in\mathcal P_1}V(P)\right)
 \le2\cdot|\mathcal P_1|\le28a.
\]
Let $G_2:=G\setminus\left(Z\cup\bigcup_{P\in\mathcal P_1}V(P)
 \cup V(H_0)\right)$. Thus
\begin{align*}
 \chi(G_2)&\ge\chi(G_1)-28a-\,\DPnew{98C_{\ref{lem:linkage}}a}\\
 &\ge C_{\ref{lem:DP-woven}}\cdot(t+227\cdot a\cdot(1+f_{7.1}(G,t))).
\end{align*}

By Theorem~\ref{lem:connectivity} applied to $G_2$, there exists a
$(C_{\ref{lem:DP-woven}}\cdot a)$-connected subgraph $G_3$ of $G_2$ with
\[
 \chi(G_3)\ge\chi(G_2)-6\cdot C_{\ref{lem:DP-woven}}\cdot a
 \ge C_{\ref{lem:DP-woven}}\cdot(t+221\cdot a\cdot(1+f_{7.1}(G,t))).
\]
First suppose that $G_3$ is $(a,0)$-woven. Since $G_1$ is
$2a$-connected, it follows from Theorem~\ref{lem:menger} that there exists
a set $\mathcal P_2$ of pairwise vertex-disjoint $V(G_3)-V(H_0)$
paths in $G_1$ with $|\mathcal P_2|=2a$.

Given the existence of $\mathcal P_1$ and $\mathcal P_2$, it
follows from  Lemma~\ref{lem:redundancy} that there exists a set
$\mathcal P$ of $8a$ pairwise vertex-disjoint
$(Z\cup V(G_3))-V(H_0)$ paths where every vertex in $Z$ is an end
of some path in $\mathcal P$.

For $i\in[7a]$, let $P_i$ be the path in $\mathcal P$ such that
$z_i$ is an end of $P_i$. Let
$\mathcal P\setminus\{P_i:i\in[7a]\}=\{P_{7a+1},\ldots,P_{8a}\}$.
For each $i\in\{7a+1,\ldots,8a\}$, let $v_i$ be the end of $P_i$
in $V(G_3)$. For each $i\in[8a]$, let $u_i\in V(H_0)$ be the other
end of $P_i$. By definition of $(a,0)$-woven, there exists a $K_a$
model $\mathcal M'$ in $G_3$ rooted on $\{v_{7a+1},\ldots,v_{8a}\}$.

For each $i\in[a]$, let $M_i'$ denote the subgraph in
$\mathcal M'$ containing $v_{7a+i}$.

Since $H_0$ is $(C_{\ref{lem:linkage}}\cdot4a)$-connected, we have by Theorem~\ref{lem:linkage}
that there exists a $\{(u_i,u_{4a+i})\}_{i\in[4a]}$-linkage
$\mathcal Q$ in $H_0$. For each $i\in[a]$, let
\[
 M_i:=P_{3a+i}\cup Q_{3a+i}\cup P_{7a+i}\cup M_i'.
\]
Then $\bigcup_{i\in[a]}M_i$ is $K_a$ model $\mathcal M$ rooted at
$R$. For each $i\in[3a]$, let
\[
 P_i':=P_i\cup Q_i\cup P_{4a+i}.
\]
Now $\bigcup_{i\in[3a]}P_i'$ is an $S-T$ linkage disjoint from
$\mathcal M$. Since $R,S,T$ were arbitrary, we find that $G$ is
$(a,3a)$-woven as desired.
Thus it suffices to show that $G_3$ is $(a,0)$-woven. To that end,
let $R'=\{r_1',\ldots,r_a'\}$ be a set of $a$ distinct vertices in $G_3$.
Since $G_3$ is $(C_{\ref{lem:DP-woven}}\cdot a)$-connected, we have that
$\delta(G_3)\ge C_{\ref{lem:DP-woven}}\cdot a\ge3a=3|R'|$ since $C_{\ref{lem:DP-woven}}\ge3$.
Thus there exists a subset $S'=\{s_1',\ldots,s_{2a}'\}$ of
$V(G_3)\setminus R'$ with $|S'|=2a$, where $s_i'$ and $s_{a+i}'$
are adjacent to $r_i'$ for every $i\in[a]$.
Let $G_4:=G_3\setminus(R'\cup S')$. So
\[
 \chi(G_4)\ge\chi(G_3)-3a
 \ge C_{\ref{lem:DP-woven}}\cdot(t+218\cdot a\cdot(1+f_{7.1}(G,t))).
\]

Let \DPnew{$m:=14a
(1+f_{7.1}(G,t))$.}
It follows that $\chi(G_4)\ge2m$.
\begin{DPaddition}
We have
\[
 \frac{14t}{\sqrt{\log(14t)}}
 <\frac{14t}{\sqrt{\log t}}<14a\le14t.
\]
\end{DPaddition}
Then since
\DPnew{$G_4$ is induced in $G$}
and $G_4$ is $K_{14a}$-minor-free, we have by
\DPnew{the definition of $f_{7.1}(G,t)$}
that $G_4$ is $m$-chromatic-separable.

Thus by definition of $m$-chromatic-separable, there exist two
vertex-disjoint subgraphs $H_1,H_2$ of $G_4$ such that for each
$i\in\{1,2\}$, we have
\[
 \chi(H_i)\ge\chi(G_4)-m
 \ge C_{\ref{lem:DP-woven}}\cdot(t+204\cdot a\cdot(1+f_{7.1}(G,t))).
\]

Similarly since $\chi(H_1)\ge2m$
and $H_1$ is $K_{14a}$-minor-free, it also follows from
\DPnew{the definition of $f_{7.1}(G,t)$}
that $H_1$ is $m$-chromatic-separable.

Hence by definition of
$m$-chromatic-separable, there exist two vertex-disjoint subgraphs
$H_{1,1},H_{1,2}$ of $H_1$ such that for each $j\in\{1,2\}$, we have
\[
 \chi(H_{1,j})\ge\chi(H_1)-m
 \ge C_{\ref{lem:DP-woven}}\cdot(t+190\cdot a\cdot(1+f_{7.1}(G,t))).
\]
Let $J_1:=H_{1,1}$, $J_2:=H_{1,2}$ and $J_3:=H_2$. Note that for
each $i\in[3]$, $\chi(J_i)\ge7\cdot C_{\ref{lem:DP-woven}}\cdot a$. Hence for
each $i\in[3]$, by Theorem~\ref{lem:connectivity} as $\chi(J_i)\ge7\cdot C_{\ref{lem:DP-woven}}\cdot a$,
there exists a $(C_{\ref{lem:DP-woven}}\cdot a)$-connected subgraph $J_i'$ of $J_i$ with
\[
 \chi(J_i')\ge\chi(J_i)-6\cdot C_{\ref{lem:DP-woven}}\cdot a\ge
 \DPnew{C_{\ref{lem:DP-woven}}\cdot}
 (t+184\cdot a\cdot(1+f_{7.1}(G,t))).
\]

Let $a':=2a/3$. Thus $276\cdot a'=184\cdot a$.

Hence for each $i\in[3]$, we have that
\[
 \chi(J_i')\ge C_{\ref{lem:DP-woven}}\cdot(t+276\cdot a'\cdot(1+f_{7.1}(G,t))).
\]

Thus by induction on $a$, we find that $J_i'$ is $(a',3a')$-woven
for each $i\in[3]$.
For each $i\in[3]$, since $\v(J_i')\ge t$, there exists a subset
$T_i=\{t'_{(i-1)a'+1},\ldots,t'_{ia'}\}$ of $V(J_i')$.
Let $T':=\{t_i':i\in[2a]\}$. Note that $|S'|=|T'|=2a$.
Since $G_3\setminus R'$ is $(C_{\ref{lem:linkage}}\cdot(2a))$-connected since
$C_{\ref{lem:DP-woven}}\ge2C_{\ref{lem:linkage}}+1$, we have by Theorem~\ref{lem:linkage} that there exists
an $\{(s_i',t_i')\}_{i\in[2a]}$-linkage $\mathcal P'$ in
$G_3\setminus R'$.

\begin{dpclaim}\label{clm:DP-weave}
There exists an $\{(s_i',t_i')\}_{i\in[2a]}$-linkage
$\mathcal P''$ in $G_3\setminus R'$ and $K_{a'}$ models $\mathcal M_i$
in $J_i'$ rooted at $T_i$ for $i\in[3]$ such that
$V(\mathcal P'')\cap V(\mathcal M_i)=T_i$.
\end{dpclaim}
\begin{proof}
Since each $J_i'$ is $(a',3a')$-woven, the claim follows by
iteratively applying Lemma~\ref{lem:weaving} to each $J_i'$.
\end{proof}

By Claim~\ref{clm:DP-weave}, there exist $\mathcal P''$ and $K_{a'}$ models
$\mathcal M_i$ as in the claim. For each $i\in[2a]$, let $P_i''$
be the path in $\mathcal P''$ containing $\{s_i',t_i'\}$.
For each $i\in[3]$ and $j\in[a']$, let $M'_{(i-1)a'+j}$ denote
the subgraph in $\mathcal M_i$ containing $t'_{(i-1)a'+j}$.
For each $i\in[a]$, let
\[
 M_i:=M_i'\cup P_i''\cup\{r_i's_i',r_i's_{a+i}'\}
       \cup P_{a+i}''\cup M_{a+i}'.
\]

Now $\mathcal M:=\{M_i:i\in[a]\}$ is a $K_a$ model in $G_3$
rooted at $R'$. Thus $G_3$ is $(a,0)$-woven as desired.
\end{proof}

\subsection{Deriving Theorem~\ref{thm:DPinsep} from Lemma~\ref{lem:DP-woven}}


Finally, we prove Theorem~\ref{thm:DPinsep}. Its derivation from Lemma~\ref{lem:DP-woven} is a fairly short and straightforward bookkeeping paralleling the derivation of~\cite[Theorem 1.6]{DP} from~\cite[Theorem 7.1]{DP}. We continue to indicate the changes in  \DPnew{dark green}.

\begin{proof}[Proof of Theorem~\ref{thm:DPinsep}]
Let $C_{\mathrm{sep}}:=3^9\cdot C_{\ref{lem:DP-woven}}$.

Suppose not. Let $t':=3^{\lceil\log_3t\rceil}$, that is the smallest power of
three that is at least as large as $t$. Thus $t'\le3t$.

By Theorem~\ref{lem:connectivity} with $k=C_{\ref{lem:DP-woven}}\cdot t'$
as $\chi(G)\ge21\cdot C_{\ref{lem:DP-woven}}\cdot t$, there exists a
$(C_{\ref{lem:DP-woven}}\cdot t')$-connected \DPnew{induced }subgraph
$G'$ of $G$ with
\[
 \chi(G')\ge\chi(G)-18\cdot C_{\ref{lem:DP-woven}}\cdot t
 \ge\DPnew{831}\cdot C_{\ref{lem:DP-woven}}\cdot t
       \cdot(1+f(G,t)),
\]
since $f(G,t)\ge0$ and
$3^9\ge\DPnew{849}$. Since $t'\le3t$, we find that
\[
 \chi(G')\ge C_{\ref{lem:DP-woven}}\cdot t'
 \cdot(1+276\cdot\DPnew{(1+}f(G,t)\DPnew{)}).
\]

We claim that $f(G,t)\ge f_{7.1}(G',t')$.
To see this, let $H\subseteq G'$ and \DPnew{real numbers $a$
and $s\ge0$ be such that 
$14t'\ge a >\frac{14t'}{\sqrt{\log(14t')}}, h(H) < a$,  and
$H$ is induced in $G'$ and
$s$-chromatic-inseparable.}
First suppose $a\le t$. Then since
$\DPnew{a>\frac{14t'}{\sqrt{\log(14t')}}\ge
             \frac{t}{\sqrt{\log t}}}$, we find that
$\DPnew{\frac sa}\le f(G,t)$
as desired. So we assume that $a>t$. Yet $a\le14t'\le42t$.
Since $H$ is $K_t$-minor-free as $G$ is, it follows that
$\DPnew{\frac sa\le\frac st}
 \le f(G,t)$ as desired.
This \DPnew{holds for every admissible $(H,a,s)$, so taking the supremum }
proves the claim.

Thus
\[
 \chi(G')\ge C_{\ref{lem:DP-woven}}\cdot t'
 \cdot(1+276\cdot\DPnew{(1+}f_{7.1}(G',t')\DPnew{)}).
\]
It now follows from Lemma~\ref{lem:DP-woven} with $a=t'$ that
$G'$ is $(t',3t')$-woven and hence contains a $K_{t'}$ minor,
contradicting that $G$ has no $K_t$ minor.
\end{proof}

\endgroup

\end{document}